\documentclass[abstract=on, 11pt, a4paper]{scrartcl}

\usepackage[T1]{fontenc}

\usepackage{amsmath, amssymb, amsthm}
\usepackage{mathtools}
\usepackage{latexsym}

\usepackage{amsbsy}

\usepackage[automark, headsepline, footsepline, plainfootsepline]{scrlayer-scrpage}

\clearscrheadfoot

\clearscrplain
\ofoot[\pagemark]{\pagemark}
\ohead{\headmark}
\automark{section}

\usepackage[left=3cm, right=2.5cm, top=2.5cm, bottom=3.5cm]{geometry}

\usepackage{enumitem}
\usepackage{autobreak}

\usepackage{xcolor}

\usepackage[most]{tcolorbox}
\tcbuselibrary{listings,theorems}

\usepackage{tikz}
\usetikzlibrary{circuits, intersections}
\usetikzlibrary{plotmarks}
\usetikzlibrary{angles,quotes,babel}

\usetikzlibrary{calc, decorations.markings, arrows.meta}
\tikzset{
	dynkinnode/.style={draw, fill, circle, inner sep=0pt, minimum size=2.5mm}, 
	every label/.style={font=\scriptsize}, 
	double line/.style={ 
		decoration={ 
			markings,
			mark=at position 0.6 with {\arrow[scale=0.65]{Straight Barb}} 
		},
		postaction={decorate},
		double, 
		double distance = 1mm
	},
	arrow sign/.style={ 
		decoration={
			markings,
			mark=at position 0.6 with {\arrow[scale=2]{Straight Barb}}
		},
		postaction={decorate},
	},
	x = 1.25cm  
}

\newcommand{\doubleline}[2]{
	\draw ($ (#1) + (0, 0.5mm) $) -- ($ (#2) + (0, 0.5mm) $);
	\draw ($ (#1) + (0, -0.5mm) $) -- ($ (#2) + (0, -0.5mm) $);
	\path[arrow sign] (#1) -- (#2);
}

\usepackage{graphicx}
\usepackage{pdfpages}

\usepackage[super]{nth}

\newcounter{countercheck}[subsection]
\renewcommand\thecountercheck{(\arabic{section}.\arabic{countercheck})}

\theoremstyle{plain}
\swapnumbers
\newtheorem{theorem}[countercheck]{Theorem}
\newtheorem{proposition}[countercheck]{Proposition}
\newtheorem{lemma}[countercheck]{Lemma}
\newtheorem{corollary}[countercheck]{Corollary}

\theoremstyle{definition}

\newtheorem{convention}[countercheck]{Convention}

\theoremstyle{remark}
\newtheorem{remark}[countercheck]{Remark}

\renewcommand{\phi}{\varphi}
\renewcommand{\epsilon}{\varepsilon}

\newcommand{\NN}{\mathbb{N}}

\newcommand{\FF}{\mathbb{F}}

\newcommand{\tcol}[1]{\textcolor{red}{#1}}

\newcommand{\KK}{\mathbb{K}}

\newcommand{\C}{\mathcal{C}}

\DeclareMathOperator{\id}{id}

\DeclareMathOperator{\proj}{proj}

\DeclareMathOperator{\Stab}{Stab}

\DeclareMathOperator{\costar}{\textup{(Co$^{\star}$)}}
\DeclareMathOperator{\nc}{\textup{(nc)}}

\numberwithin{equation}{section} 
\title{On commutator relations in $2$-spherical RGD-systems}
\author{Sebastian \textit{Bischof}\footnote{email: sebastian.bischof@math.uni-giessen.de} \\
	Mathematisches Institut, Arndtstra\ss e 2, 35392 Gie\ss en, Germany}

\date{\today}

\begin{document}

\maketitle

\begin{abstract}
	In this paper we investigate the commutator relations for prenilpotent roots which are nested. These commutator relations are trivial in a lot of cases.
\end{abstract}

\section{Introduction}

In \cite{Ti92} Tits introduced RGD-systems in order to describe groups of Kac-Moody type. Let $(W, S)$ be a Coxeter system, let $\Phi$ be the associated set of roots (viewed as a set of half-spaces) and let $\Pi$ be a basis (i.e. a set of simple roots) of $\Phi$. An RGD-system of type $(W, S)$ is a pair $(G, (U_{\alpha})_{\alpha \in \Phi})$ consisting of a group $G$ and a family of subgroups $U_{\alpha}$ (called \textit{root groups}) indexed by the set of roots $\Phi$ satisfying a few axioms (for the precise definition see Section \ref{sec:rgd}). In this paper we are interested in \textit{$2$-spherical} RGD-systems (i.e. the order of $st$ in $W$ is finite for any $s, t \in S$). For $\alpha \neq \beta \in \Pi$ we define $X_{\alpha} := \langle U_{\alpha} \cup U_{-\alpha} \rangle$ and $X_{\alpha, \beta} := \langle X_{\alpha} \cup X_{\beta} \rangle$. A $2$-spherical RGD-system satisfies Condition $\costar$ if the following holds:
\begin{equation}
	X_{\alpha, \beta} / Z(X_{\alpha, \beta}) \not\cong B_2(2), G_2(2), G_2(3), ^2 F_4(2) \text{ for all pairs } \{ \alpha, \beta \} \subseteq \Pi. \tag{Co$^{\star}$}
\end{equation}
This condition was introduced by M\"uhlherr and Ronan in \cite{MR95}. In loc.cit. they give a formulation in terms of buildings.

One axiom of an RGD-system provides a commutator relation between $U_{\alpha}$ and $U_{\beta}$, where $\{ \alpha, \beta \} \subseteq \Phi$ is a pair of \textit{prenilpotent roots} (i.e. a pair of distinct roots where both $\alpha \cap \beta$ and $(-\alpha) \cap (-\beta)$ are non-empty sets). In the spherical case two roots $\alpha, \beta$ are prenilpotent if and only if $\alpha \neq -\beta$. In the non-spherical case this is no longer true. For a pair $\{ \alpha, \beta \}$ of prenilpotent roots we have either $o(r_{\alpha} r_{\beta}) <\infty$ or $o(r_{\alpha} r_{\beta}) = \infty$. In the first case the commutator relations between $U_{\alpha}$ and $U_{\beta}$ are completely determined by Tits and Weiss in \cite{TW02}. In the second case the pair $\{ \alpha, \beta \}$ is \textit{nested}, i.e. $\alpha \subsetneq \beta$ or $\beta \subsetneq \alpha$. In this paper we are interested in the commutator relations of root groups corresponding to nested roots in $2$-spherical RGD-systems of rank $3$. We will show that in a certain class of RGD-systems these commutator relations are trivial. This leads to the following definition:

An RGD-system satisfies Condition $\nc$, if the following holds:
\begin{equation}
	\forall \alpha \subsetneq \beta \in \Phi: [U_{\alpha}, U_{\beta}] = 1. \tag{nc}
\end{equation}

We prove the following theorem (cf. Proposition \ref{Propsimplylacedaffine} and Theorem \ref{Mainresult}):

\medskip
\noindent
\textbf{Theorem A:} Let $(W, S)$ be a $2$-spherical Coxeter system of rank $3$ and let $\mathcal{D} = (G, (U_{\alpha})_{\alpha \in \Phi})$ be an RGD-system of type $(W, S)$ satisfying Condition $\costar$. If $\mathcal{D}$ does not satisfy Condition $\nc$ then $\mathcal{D}$ is of type $(2, 4, 4), (2, 4, 6), (2, 4, 8), (2, 6, 6)$ or $(2, 6, 8)$.

\medskip
\noindent The following proposition shows that our theorem is optimal in the sense that there are RGD-systems for any type mentioned in Theorem A, where Condition $\costar$ does not imply Condition $\nc$ (see Remark \ref{ExampleC2tilde} and Section \ref{sec:Non-cyclic}):

\medskip
\noindent
\textbf{Proposition B:} There exist RGD-systems of type $(2, 4, 4), (2, 4, 6), (2, 4, 8), (2, 6, 6), (2, 6, 8)$, which satisfy Condition $\costar$, but do not satisfy Condition $\nc$.

\medskip
\noindent In the higher rank case we have the following result which is contained in Theorem \ref{Mainresult}:

\medskip
\noindent
\textbf{Theorem C:} Assume that the Coxeter diagram of $(W, S)$ is the complete graph or simply-laced (i.e. $m_{st} \in \{2, 3\}$ for all $s, t\in S$). Then in every RGD-system of type $(W, S)$ Condition $\costar$ implies Condition $\nc$.

\medskip
\noindent \textbf{Remarks:}

\smallskip
\noindent $1$. We believe that we can't drop Condition $\costar$ in general in Theorem A. But if the rank $3$ RGD-system is of affine type, we can drop Condition $\costar$ (see Proposition \ref{Propsimplylacedaffine}). Furthermore, Theorem A shows that Condition $\nc$ is a weaker condition than Condition $\costar$ for a $2$-spherical RGD-system of rank $3$ which is not of type $(2, 4, 4), (2, 4, 6), (2, 4, 8), (2, 6, 6), (2, 6, 8)$.

\smallskip
\noindent $2$. Let $(W, S)$ be of type $(k, l, m)$ and we assume that $2 \leq k \leq l \leq m \leq 8$. If $l \in \{3, 8\}$, then in an RGD-system of type $(W, S)$ Condition $\costar$ implies Condition $\nc$. Moreover, there are many $2$-spherical non-affine diagrams $(W, S)$ such that in an RGD-system of type $(W, S)$ Condition $\costar$ does not imply Condition $\nc$. Since $k, l, m \in \{2, 3, 4, 6, 8\}$, there are $125$ different types of RGD-systems of rank $3$ and in only five types Condition $\costar$ does not necessarily imply Condition $\nc$.

\smallskip
\noindent $3$. In Theorem C we can drop Condition $\costar$, if $(W, S)$ is simply-laced, since every simply-laced RGD-system automatically satisfies Condition $\costar$. Theorem C is proved by Allcock and Carbone for RGD-systems of simply-laced hyperbolic type associated to Kac-Moody groups over fields (see Lemma $6$ in \cite{AC16}). Our theorem includes these cases. Furthermore, we learned that Ted Williams produced results in \cite{Wi20} that are similar to our Theorem A and Proposition B. We also remark, that Condition $\nc$ often holds for RGD-systems, which come from Kac-Moody groups over fields (cf. Remark $3.7 (f)$ in \cite{Ti87}).

\renewcommand{\abstractname}{Acknowledgement}
\begin{abstract}
	The author thanks the anonymous referee for careful reading any many helpful comments.
\end{abstract}

\section{Preliminaries}

\subsection*{Coxeter systems}

Let $(W, S)$ be a Coxeter system and let $\ell$ denote the corresponding length function. Defining $w \sim_s w'$ if and only if $w^{-1}w' \in \langle s \rangle$ we obtain a chamber system with chamber set $W$ and equivalence relations $\sim_s$ for $s\in S$, which we denote by $\Sigma(W, S)$. We call two chambers $w, w'$ \textit{$s$-adjacent} if $w \sim_s w'$ and \textit{adjacent} if they are $s$-adjacent for some $s\in S$. A \textit{gallery of length $n$} from $w_0$ to $w_n$ is a sequence $(w_0, \ldots, w_n)$ of chambers where $w_i$ and $w_{i+1}$ are adjacent for any $0 \leq i < n$. A gallery $(w_0, \ldots, w_n)$ is called \textit{minimal} if there exists no gallery from $w_0$ to $w_n$ of length $k<n$ and we denote the length of a minimal gallery between $w_0$ and $w_n$ by $d(w_0, w_n)$. For $s, t \in S$ we denote the order of $st$ in $W$ by $m_{st}$. The \textit{Coxeter diagram} corresponding to $(W, S)$ is the labeled graph $(S, E(S))$, where $E(S) = \{ \{s, t \} \mid m_{st}>2 \}$ and where each edge $\{s,t\}$ is labeled by $m_{st}$ for all $s, t \in S$. The \textit{rank} of a Coxeter diagram is the cardinality of the set of its vertices. It is well-known that the pair $(\langle J \rangle, J)$ is a Coxeter system (cf. \cite[Ch. IV, §$1$ Th\'eor\`eme $2$]{Bo68}). For $J \subseteq S$ we define the \textit{$J$-residue} of a chamber $c\in W$ to be the set $c \langle J \rangle$. A \textit{residue} $R$ is a $J$-residue for some $J \subseteq S$; we call $J$ the \textit{type} of $R$ and the cardinality of $J$ is called the \textit{rank} of $R$. A \textit{panel} is a residue of rank $1$. It is a fact that for any chamber $x\in W$ and every residue $R$ there exists a unique chamber $z\in R$ such that $d(x, y) = d(x, z) + d(z, y)$ holds for any chamber $y\in R$. The chamber $z$ is called the \textit{projection} of $x$ onto $R$ and is denoted by $z = \proj_R x$.

Let $(W, S)$ be a $2$-spherical Coxeter system of rank $3$ and let $S = \{ r, s, t \}$. Sometimes we will also call $(m_{rs}, m_{rt}, m_{st})$ the \textit{type} of $(W, S)$. The Coxeter system $(W, S)$ is called \textit{affine} if $\frac{1}{m_{rs}} + \frac{1}{m_{rt}} + \frac{1}{m_{st}} =1$, and it is called \textit{hyperbolic} if $\frac{1}{m_{rs}} + \frac{1}{m_{rt}} + \frac{1}{m_{st}} <1$. A Coxeter system of rank $3$ is called \textit{cyclic} if the underlying Coxeter diagram is the complete graph.

\subsection*{Roots and walls}

A \textit{reflection} is an element of $W$ that is conjugated to an element of $S$. For $s\in S$ we let $\alpha_s := \{ w\in W \mid \ell(sw) > \ell(w) \}$ be the \textit{simple root} corresponding to $s$. A \textit{root} is a subset $\alpha \subseteq W$ such that $\alpha = v\alpha_s$ for some $v\in W$ and $s\in S$. We denote the set of all roots by $\Phi(W, S)$. The set $\Phi(W, S)_+ = \{ \alpha \in \Phi(W, S) \mid 1_W \in \alpha \}$ is the set of all \textit{positive roots} and $\Phi(W, S)_- = \{ \alpha \in \Phi(W, S) \mid 1_W \notin \alpha \}$ is the set of all \textit{negative roots}. For each root $\alpha \in \Phi(W, S)$ we denote the \textit{opposite root} by $-\alpha$ and we denote the unique reflection which interchanges these two roots by $r_{\alpha}$.

For a pair $\{ \alpha, \beta \}$ of prenilpotent roots we will write $\left[ \alpha, \beta \right] := \{ \gamma \in \Phi(W, S) \mid \alpha \cap \beta \subseteq \gamma \text{ and } (-\alpha) \cap (-\beta) \subseteq -\gamma \}$ and $(\alpha, \beta) := \left[ \alpha, \beta \right] \backslash \{ \alpha, \beta \}$.

For $\alpha \in \Phi(W, S)$ we denote by $\partial \alpha$ (resp. $\partial^2 \alpha$) the set of all panels (resp. spherical residues of rank $2$) stabilized by $r_{\alpha}$. Furthermore, we define $\mathcal{C}(\partial^2 \alpha) := \bigcup_{R \in \partial^2 \alpha} R$. Some authors call a panel $P \in \partial \alpha$ a \textit{boundary panel} of the root $\alpha$.

The set $\partial \alpha$ is called the \textit{wall} associated to $\alpha$. Let $G = (c_0, \ldots, c_k)$ be a gallery. We say that $G$ \textit{crosses the wall $\partial \alpha$} if there exists $1 \leq i \leq k$ such that $\{ c_{i-1}, c_i \} \in \partial \alpha$. It is a basic fact that a minimal gallery crosses a wall at most once (cf. Lemma $3.69$ in \cite{AB08}).

\begin{convention}
	For the rest of this paper we let $(W, S)$ be a $2$-spherical Coxeter system of finite rank and $\Phi := \Phi(W, S)$ (resp. $\Phi_+$ and $\Phi_-$).
\end{convention}

\subsection*{Reflection triangles and combinatorial triangles}

This subsection is based on \cite{CM05}.

A set $D$ a called \textit{fundamental domain} for the action of a group $G$ on a set $E$ containing $D$ if $\bigcup_{ g\in G } gD = E$ and $D \cap gD \neq \emptyset \Rightarrow g=1$ for any $g\in G$.

Let $\Psi$ be a set of roots. We put $R(\Psi) := \{ r_{\psi} \mid \psi \in \Psi \}$ and $W(\Psi) := \langle R(\Psi) \rangle$. The set $\Psi$ is called \textit{geometric} if $\bigcap_{\psi \in \Psi} \psi$ is non-empty and if for all $\phi, \psi \in \Psi$, the set $\phi \cap \psi$ is a fundamental domain for the action of $W(\{ \phi, \psi \})$ on $\Sigma(W, S)$.

A \textit{parabolic subgroup} of $W$ is a subgroup of the form $\Stab_W(R)$ for some residue $R$. The \textit{rank} of the parabolic subgroup is defined to be the rank of the residue $R$. A \textit{reflection triangle} is a set $T$ of three reflections such that the order of $tt'$ is finite for all $t, t' \in T$ and such that $T$ is not contained in any parabolic subgroup of rank $2$. It is a fact that given a reflection triangle $T$ there exists a reflection triangle $T'$ such that $\langle T \rangle = \langle T' \rangle$ and that $(\langle T \rangle, T')$ is a Coxeter system. This follows essentially from Proposition $4.2$ in \cite{MW02}. In particular, the Coxeter diagram of $(\langle T \rangle, T')$ is uniquely determined by $\langle T \rangle$ and we denote this diagram by $\mathcal{M}(T)$. 
We say that $T$ is \textit{affine} if $\mathcal{M}(T)$ is affine. A set of three roots $T$ is called \textit{combinatorial triangle} (or simply a \textit{triangle}) if the following hold:
\begin{enumerate}[label=(CT\arabic*), leftmargin=*]
	\item The set $\{ r_{\alpha} \mid \alpha \in T \}$ is a reflection triangle.
	
	\item The set $\bigcap_{\alpha \in T} \alpha$ is geometric.
\end{enumerate}

A triangle $T = \{ \alpha_0, \alpha_1, \alpha_2 \}$ is called a \textit{fundamental triangle} if $(-\alpha_i, \alpha_{i+1}) = \emptyset$ for any $0 \leq i \leq 2$, where the indices are taken modulo $3$. To avoid complicated subscripts we will write $r_i$ instead of $r_{\alpha_i}$ for $i\in \NN$ and a root $\alpha_i \in \Phi$.

\begin{lemma}\label{Theorem1.2CM}
	Let $T$ be an affine reflection triangle. Then there exists an irreducible affine parabolic subgroup $W_0 \leq W$ of rank $\geq 3$ such that $\langle T \rangle$ is conjugated to a subgroup of $W_0$.
\end{lemma}
\begin{proof}
	This is Theorem $1.2$ of \cite{CM05}, which is essentially based on a result in \cite{KDis}.
\end{proof}

Let $T = \{ \alpha_0, \alpha_1, \alpha_2 \}$ be a triangle and let $\sigma_i \in \partial^2 \alpha_{i-1} \cap \partial^2 \alpha_{i+1}$ (the indices are taken modulo $3$) which is contained in $\alpha_i$. The set $\{ \sigma_0, \sigma_1, \sigma_2 \}$ is called a \textit{set of vertices} of $T$. For $0 \leq i \neq j \leq 2$ we let $x_{i, j}$ be the unique chamber of $\proj_{\sigma_i} \sigma_j$ which belongs to $\alpha_k$ where $i \neq k \neq j$. Using Lemma $2.3$ of \cite{CM05}, there exists for each $0 \leq i \leq 2$ a minimal gallery $\Gamma_i$ from $x_{i-1, i+1}$ to $x_{i+1, i-1}$ such that every element belongs to $\C(\partial^2 \alpha_i) \cap \alpha_i$. Let also $\tilde{\Gamma}_i$ be the unique minimal gallery from $x_{i, i+1}$ to $x_{i, i-1}$. Finally, let $\Gamma$ be the gallery obtained by concatenating the $\Gamma_i$'s and the $\tilde{\Gamma}_i$'s, i.e.
\[ \Gamma = \Gamma_0 \sim \tilde{\Gamma}_1 \sim \Gamma_2 \sim \tilde{\Gamma}_0 \sim \Gamma_1 \sim \tilde{\Gamma}_2. \]
Then $\Gamma$ is a closed gallery by construction and we say that $\Gamma$ \textit{skirts around} the triangle $T$ and that the vertices $\{ \sigma_0, \sigma_1, \sigma_2 \}$ \textit{supports} $\Gamma$. The \textit{perimeter} of $T$ is the minimum of the set of all lengths of all galleries that skirt around $T$. In the next lemma we use the same induction as in Lemma $5.2$ of \cite{CM05} but we deduce different things.

\begin{proposition}\label{completefundamentaltriangle}
	Assume that the Coxeter diagram is the complete graph and let $T$ be a triangle. Then $T$ is fundamental.
\end{proposition}
\begin{proof}
	Let $T = \{ \alpha_0, \alpha_1, \alpha_2 \}, \{ 0, 1, 2 \} = \{ i, j, k \}$, let $p$ be the perimeter of $T$, let $\Gamma$ be a closed gallery of length $p$ which skirts around $T$ and let $\{ \sigma_0, \sigma_1, \sigma_2 \}$ be a set of vertices that supports $\Gamma$. We prove the hypothesis by induction on the length $p$. If $p=0$ then $\sigma_0, \sigma_1, \sigma_2$ contain a common chamber. Then $\bigcap_{\alpha \in T} \alpha$ is a chamber and hence $T$ is fundamental. Thus we assume $p>0$. Assume that $T$ is not fundamental and that $\beta \in (-\alpha_i, \alpha_j)$. Then $o(r_{\beta} r_k) < \infty$ and we obtain the reflection triangles $T_+ := \{ \alpha_i, \beta, \alpha_k \}$ and $T_- := \{ -\beta, \alpha_j, \alpha_k \}$.
	Let $\sigma \in \partial^2 \alpha_k \cap \partial^2 \beta$ such that $\sigma$ is crossed by $\Gamma$ (i.e. there exists a root $\alpha$ such that $\sigma \in \partial^2 \alpha$ and $\partial \alpha$ is crossed by $\Gamma$). Then the perimeter of $T_+$ and $T_-$ is $<p$ and we can use induction. This implies that $T_+$ and $T_-$ are fundamental triangles and hence $(\beta, \alpha_k) = \emptyset = (-\alpha_k, \beta)$. Since $m_{st} \neq 2$ for any $s \neq t \in S$ this yields a contradiction.
\end{proof}

\section{Triangles in the hyperbolic plane}

In this section we consider $\Sigma(W, S)$, where $(W, S)$ is hyperbolic and of rank $3$. In this case $\Sigma(W, S)$ has a geometric realization whose underlying metric space is $\mathbb{H}^2$. For more details see the discussion in the proof of Theorem $14$ in \cite{CR09}.

Let $\alpha, \beta \in \Phi$ such that $o(r_{\alpha}r_{\beta}) < \infty$. Then $r_{\alpha}$ and $r_{\beta}$ have a unique common fixed point in $\mathbb{H}^2$. We denote the angle counterclockwise between $r_{\alpha}$ and $r_{\beta}$ divided by $\pi$ by $\angle r_{\alpha} r_{\beta}$:
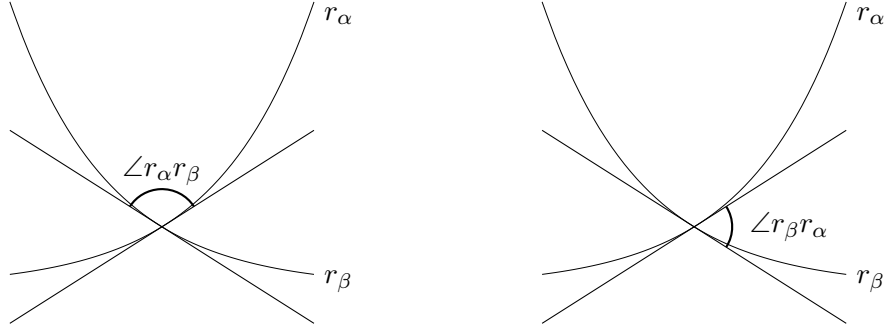
\[ \begin{tikzpicture}[scale=0.8]


\draw[line width=0.01mm, domain=-2:2, smooth] plot (\x, {pow(0.8*2.71828, \x)});
\draw (2, 4.5) node [right]{$r_{\alpha}$};

\draw[line width=0.01mm, domain=-2:2, smooth] plot (\x, {pow(0.8*2.71828, -\x)});
\draw (2, 0.1) node [right]{$r_{\beta}$};

\draw[domain=-2:2] plot (\x, 0.8*\x+1);	
\draw[domain=-2:2] plot (\x, -0.8*\x+1);

\coordinate (A) at (-2, 2.6);
\coordinate (B) at (2, 2.6);
\coordinate (M) at (0, 1);
\coordinate (C) at (2, -0.6);

\tikzset{anglestyle/.style={angle eccentricity=1.5, draw,  thick, angle radius=0.5cm}}

\draw 
pic ["$\angle r_{\alpha} r_{\beta}$", anglestyle] {angle = B--M--A};


\draw[line width=0.01mm, domain=5:9, smooth] plot (\x, {pow(0.8*2.71828, \x-7)});
\draw (9, 4.5) node [right]{$r_{\alpha}$};

\draw[line width=0.01mm, domain=5:9, smooth] plot (\x, {pow(0.8*2.71828, -\x+7)});
\draw (9, 0.1) node [right]{$r_{\beta}$};

\draw[domain=5:9] plot (\x, 0.8*\x -5.6+1);	
\draw[domain=5:9] plot (\x, -0.8*\x +5.6+1);

\coordinate (A1) at (5, 2.6);
\coordinate (B1) at (9, 2.6);
\coordinate (M1) at (7, 1);
\coordinate (C1) at (9, -0.6);

\tikzset{anglestyle/.style={angle eccentricity=1.5, draw,  thick, angle radius=0.5cm}}

\draw 
pic ["$\angle r_{\beta} r_{\alpha}$", anglestyle, angle eccentricity=2.5] {angle = C1--M1--B1};

\end{tikzpicture} \]

Note that $\angle r_{\alpha} r_{\beta} + \angle r_{\beta} r_{\alpha} = 1$ and that in general $\angle r_{\alpha} r_{\beta} \neq \angle r_{\beta} r_{\alpha}$. Since we consider RGD-systems we can assume that every angle between two reflections having a unique common fixed point is in the following set (cf. \cite[$(17.1)$ Theorem]{TW02}): $\{ \frac{\pi}{8}, \frac{\pi}{6}, \frac{\pi}{4}, \frac{\pi}{3}, \frac{3\pi}{8}, \frac{\pi}{2}, \frac{5\pi}{8}, \frac{2\pi}{3}, \frac{3\pi}{4}, \frac{5\pi}{6}, \frac{7\pi}{8} \}$. In $\mathbb{H}^2$ every triangle has interior angle sum strictly less than $\pi$ and every quadrangle has interior angle sum strictly less than $2\pi$. It is a fact that two different lines of $\mathbb{H}^2$ intersect in at most one point of $\mathbb{H}^2$. Thus for two roots $\alpha, \beta \in \Phi$ such that $o(r_{\alpha} r_{\beta}) < \infty$ there exists a unique residue $R$ of rank $2$ which is stabilized by $r_{\alpha}$ and $r_{\beta}$. Let $\{ s, t \}$ be the type of $R$. Then we define $m_{\alpha, \beta} := m_{st}$ and we remark that $o(r_{\alpha} r_{\beta})$ divides $m_{\alpha, \beta}$.

\begin{corollary}\label{uniquechamberhyperbolicrank3}
	Every triangle contains a unique chamber, if the Coxeter diagram is cyclic hyperbolic.
\end{corollary}
\begin{proof}
	This is a consequence of the classification in \cite{Fe98} (cf. Figure $8$ in $\S 5.1$ in loc.cit).
\end{proof}

\begin{remark}
For a fundamental triangle $T = \{ \alpha, \beta, \gamma \}$ in a Coxeter system of hyperbolic type and of rank $3$, the Coxeter system is of type $(m_{\alpha, \beta}, m_{\alpha, \gamma}, m_{\beta, \gamma})$ (cf. \cite{Fe98} and the previous corollary). In particular, $\bigcap_{\epsilon \in T} \epsilon$ is a fundamental domain for the action of $W$ on the geometric realization of $\Sigma(W, S)$.
\end{remark}

\begin{lemma}\label{angle1over2}
Let $T = \{ \alpha_0, \alpha_1, \alpha_2 \}$ be a triangle and let $\{ i, j, k \} = \{ 0, 1, 2 \}$. Let $\alpha_3 \in (-\alpha_i, \alpha_j)$. Then the following hold:
\begin{enumerate}[label=(\alph*)]
	\item We have $\angle r_k r_3 \neq \frac{5}{8} \neq \angle r_3 r_k$.
	
	\item We have $\angle r_k r_3 = \angle r_3 r_k = \frac{1}{2}$ if the diagram is not of type $(2, 3, 8)$. In particular, we have $(-\alpha_i, \alpha_k) = (-\alpha_j, \alpha_k) = \emptyset$ and $(-\alpha_i, \alpha_j) = \{ \alpha_3 \}$.
\end{enumerate}
\end{lemma}
\begin{proof}
	This is also a consequence of the classification in \cite{Fe98} (cf. Figure $8$ in $\S 5.1$ in loc.cit).
\end{proof}

\begin{proposition}\label{Prop238}
Assume that the Coxeter diagram is of type $(2, 3, 8)$. Let $T = \{ \alpha_0, \alpha_1, \alpha_2 \}$ be a triangle and let $\{ i, j, k \} = \{ 0, 1, 2 \}$. If there exists a root $\alpha_3 \in (-\alpha_i, \alpha_j)$ such that the angle between $r_k$ and $r_3$ in the reflection triangle $\{ r_j, r_k, r_3 \}$ is $\frac{2}{3} \pi$, then $(-\alpha_j, \alpha_k) = (-\alpha_i, \alpha_k) = \emptyset, \vert (-\alpha_i, \alpha_j) \vert \leq 3$ and $m_{\alpha_i, \alpha_j} = 8$.
\end{proposition}
\begin{proof}
	The first part is a consequence of the classification in \cite{Fe98} (cf. Figure $8$ in $\S 5.1$ in loc.cit). Let $\alpha_3 \in (-\alpha_i, \alpha_j)$ such that the angle between $r_k$ and $r_3$ in the reflection triangle $\{ r_j, r_k, r_3 \}$ is $\frac{2}{3} \pi$. Then we have $\angle r_3 r_j = \angle r_j r_k = \frac{1}{8}$ because of the interior angle sum of a triangle. This implies $m_{\alpha_i, \alpha_j} = 8$. 
\end{proof}

\begin{lemma}\label{Prop238auxres}
	Assume that the Coxeter diagram is of type $(2, 3, 8)$. Let $T = \{ \alpha_0, \alpha_1, \alpha_2 \}$ be a triangle and let $\{ i, j, k \} = \{ 0, 1, 2 \}$. Assume that $(-\alpha_i, \alpha_k) = (-\alpha_j, \alpha_k) = \emptyset$ and that $m_{\alpha, \beta} = 8$ for any $\alpha \neq \beta \in T$. Then one of the following hold:
	\begin{enumerate}[label=(\alph*)]
		\item $\vert (-\alpha_i, \alpha_j) \vert = 1$ and there exists a root $\beta$ such that $o(r_i r_{\beta}) < \infty$ and $m_{\alpha_i, \beta} = 3$.
		
		\item $\angle r_i r_j = \frac{1}{2}$.
	\end{enumerate}
\end{lemma}
\begin{proof}
	In the following proof, up to reordering, we can assume that we are in the situation displayed in the figure. By hypothesis $m_{\alpha_i, \alpha_j} = 8$ and we obtain $\angle r_i r_j \in \{ \frac{1}{8}, \ldots, \frac{5}{8} \}$. Since the diagram is not of type $(8, 8, 8)$, it follows that $\angle r_i r_j \neq \frac{1}{8}$. Thus there exists a root $\alpha_3 \in (-\alpha_i, \alpha_j)$ such that $\angle r_i r_3 = \frac{1}{8}$. Using Lemma \ref{angle1over2} we obtain $\angle r_3 r_k \in \{ \frac{1}{3}, \frac{1}{2}, \frac{2}{3} \}$. Since the diagram is not of type $(3, 8, 8)$, we obtain $\angle r_3 r_k \neq \frac{1}{3}$. We distinguish the following two cases:
	\begin{enumerate}[label=(\alph*)]
		\item $\angle r_3 r_k = \frac{2}{3}$: Then $\angle r_k r_3 = \frac{1}{3}$. Since $(W, S)$ is not of type $(3, 8, 8)$, we obtain a root $\alpha_4 \in (\alpha_3, \alpha_j)$ such that $\angle r_3 r_4 = \frac{1}{8}$. Applying Lemma \ref{angle1over2} we have $\angle r_4 r_k \in \{ \frac{1}{3}, \frac{1}{2}, \frac{2}{3} \}$. Using the interior angle sum of a triangle and the fact that the diagram is not of type $(3, 3, 8)$, we obtain $\angle r_4 r_k = \frac{1}{2}$. Assume that $(-\alpha_k, \alpha_4) \neq \emptyset$. Then there exists a root $\gamma \in (-\alpha_k, \alpha_4)$ such that $\angle r_4 r_{\gamma} = \frac{1}{4} = \angle r_{\gamma} r_k$. But then $\angle r_3 r_{\gamma} < \frac{1}{2}$ and hence $\angle r_3 r_{\gamma} = \frac{1}{3}$. This implies $\angle r_3 r_4 + \angle r_4 r_{\gamma} + \angle r_{\gamma} r_3 >1$ which is a contradiction. Thus $(-\alpha_k, \alpha_4) = \emptyset$. This implies $(\alpha_4, \alpha_k) = \emptyset$. Since the diagram is not of type $(2, 8, 8)$, we obtain a root $\alpha_5 \in (\alpha_4, \alpha_j)$ such that $\angle r_4 r_5 = \frac{1}{8}$. Moreover, we obtain $\angle r_5 r_k = \frac{1}{3}$. Thus $\angle r_k r_5 = \frac{2}{3}$ and hence $\angle r_i r_j = \frac{1}{2}$.
		
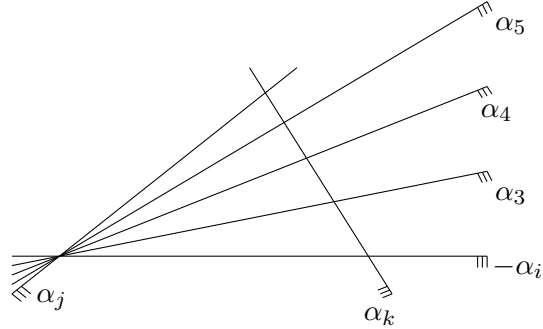
\[ \begin{tikzpicture}[scale=0.5]
		
		\draw (0  ,  0  ) to (10 ,  0  );
		\draw (10 ,  0  ) to (10 , -0.3);
		\draw (9.9,  0  ) to (9.9, -0.3);
		\draw (9.8,  0  ) to (9.8, -0.3);
		\draw (9.9, -0.3) node [right]{$-\alpha_i$};
		
		\draw (0, -1) to (6, 5);
		\draw (0  , -1  ) to (0.2, -1.2);
		\draw (0.1, -0.9) to (0.3, -1.1);
		\draw (0.2, -0.8) to (0.4, -1  );
		\draw (0.3, -1.15) node [right]{$\alpha_j$};
		
		\draw (5, 5) to (8, -1);
		\draw (8  , -1  ) to (7.8, -1.1);
		\draw (7.95, -0.9) to (7.75, -1);
		\draw (7.9, -0.8) to (7.7, -0.9);
		\draw (7.75, -1.1) node [below]{$\alpha_k$};
		
		\draw (0  , -0.25 ) to (10, 2.25);
		\draw (10  ,  2.25   ) to (10.1, 2);
		\draw (9.9,  2.225) to (10  , 1.975);
		\draw (9.8,  2.2 ) to (9.9, 1.95);
		\draw (10.5, 2.1) node [below]{$\alpha_3$};
		
		\draw (0, -0.5) to (10, 4.5);
		\draw (10, 4.5) to (10.1, 4.3);
		\draw (9.9, 4.45) to (10, 4.25);
		\draw (9.8, 4.4) to (9.9, 4.2);
		\draw (10.2, 4.3) node [below]{$\alpha_4$};
		
		\draw (0, -0.75) to (10, 6.75);
		\draw (10, 6.75) to (10.1, 6.5);
		\draw (9.9, 6.675) to (10, 6.425);
		\draw (9.8, 6.6) to (9.9, 6.35);
		\draw (10.5, 6.6) node [below]{$\alpha_5$};

	\end{tikzpicture} \]
		
		\item $\angle r_3 r_k = \frac{1}{2}$: Since the diagram is not of type $(2, 8, 8)$, we obtain three roots $\alpha_4, \alpha_5, \alpha_6 \in (\alpha_3, \alpha_k)$ such that $\angle r_k r_4 = \angle r_4 r_5 = \angle r_5 r_6 = \angle r_6 r_3 = \frac{1}{8}$. Using similar arguments as above we obtain $\angle r_j r_6 = \frac{2}{3}$ and hence $\angle r_3 r_j = \frac{1}{8}$. Thus we have $\angle r_i r_j = \frac{2}{8}$. Furthermore, we obtain three roots $\alpha_7, \alpha_8, \alpha_9 \in (-\alpha_3, \alpha_k)$ with $\angle r_3 r_7 = \angle r_7 r_8 = \angle r_8 r_9 = \angle r_9 r_k = \frac{1}{8}$. This implies $\angle r_7 r_i = \frac{2}{3}$ as above. For $\beta = \alpha_7$ the claim follows.\qedhere
	\end{enumerate}
\end{proof}

\section{Root group data}\label{sec:rgd}

An \textit{RGD-system of type $(W, S)$} is a pair $\mathcal{D} = \left( G, \left( U_{\alpha} \right)_{\alpha \in \Phi}\right)$ consisting of a group $G$ together with a family of subgroups $U_{\alpha}$ (called \textit{root groups}) indexed by the set of roots $\Phi$, which satisfies the following axioms, where $H := \bigcap_{\alpha \in \Phi} N_G(U_{\alpha}), U_{\pm} := \langle U_{\alpha} \mid \alpha \in \Phi_{\pm} \rangle$:
\begin{enumerate}[label=(RGD\arabic*), leftmargin=*] \setcounter{enumi}{-1}
	\item For each $\alpha \in \Phi$, we have $U_{\alpha} \neq \{1\}$.
	
	\item For each prenilpotent pair $\{ \alpha, \beta \} \subseteq \Phi$, the commutator group $[U_{\alpha}, U_{\beta}]$ is contained in the group $U_{(\alpha, \beta)} := \langle U_{\gamma} \mid \gamma \in (\alpha, \beta) \rangle$.
	
	\item For every $s\in S$ and each $u\in U_{\alpha_s} \backslash \{1\}$, there exists $u', u'' \in U_{-\alpha_s}$ such that the product $m(u) := u' u u''$ conjugates $U_{\beta}$ onto $U_{s\beta}$ for each $\beta \in \Phi$.
	
	\item For each $s\in S$, the group $U_{-\alpha_s}$ is not contained in $U_+$.
	
	\item $G = H \langle U_{\alpha} \mid \alpha \in \Phi \rangle$.
\end{enumerate}

\begin{lemma}\label{costarsimplerggenerate}
	Let $(G, (U_{\alpha})_{\alpha \in \Phi})$ be an RGD-system of type $(W, S)$ satisfying Condition $\costar$. Then we have $\langle U_{\gamma} \mid \gamma \in [\alpha, \beta] \rangle = \langle U_{\alpha} \cup U_{\beta} \rangle$ or equivalently $U_{(\alpha, \beta)} = [U_{\alpha}, U_{\beta}]$ for any pair of roots $\{ \alpha, \beta \} \subseteq \Phi$ which is a basis of a finite root subsystem.
\end{lemma}
\begin{proof}
	This follows from Lemma $18$ and Proposition $7$ on Page $60$ of \cite{Ab96}.
\end{proof}

\begin{lemma}\label{hexagonproperties}
	Let $\mathcal{D} = (G, (U_{\alpha})_{\alpha \in \Phi})$ be an RGD-system of type $(W, S)$.
	\begin{enumerate}[label=(\alph*)]
		\item If $\mathcal{D}$ is simply-laced or of type $G_2$, then every root group is abelian.
		
		\item If $\mathcal{D}$ is of type $G_2$ and if $\alpha, \beta \in \Phi$ such that $\vert (\alpha, \beta) \vert = 2$, then $[U_{\alpha}, U_{\beta}] = 1$.
	\end{enumerate}
\end{lemma}
\begin{proof}
	This follows from the classification in \cite{TW02}. In particular, we used Theorem $(17.1)$, Theorem $(17.5)$, Example $(16.1)$ and Example $(16.8)$ in loc.cit.
\end{proof}

\subsection*{RGD-systems of type $I_2(8)$}

In this subsection we let $\mathcal{D} = (G, (U_{\alpha})_{\alpha \in \Phi})$ be an RGD-system of type $I_2(8)$.

\begin{lemma}\label{octagonproperties}
	Let $\alpha, \beta \in \Phi$. Then the following hold:
	\begin{enumerate}[label=(\alph*)]
		\item If $\vert (\alpha, \beta) \vert = 3$, then $[U_{\alpha}, U_{\beta}] = 1$.
		
		\item If $\mathcal{D}$ satisfies Condition $\costar$, then there is only one kind of root groups which is abelian. If, furthermore, $U_{\alpha}$ is abelian and $\vert (\alpha, \beta) \vert = 1$, then $[U_{\alpha}, U_{\beta}] = 1$.
	\end{enumerate}
\end{lemma}
\begin{proof}
	This follows from Example $(16.9)$ and $(10.15)$ of \cite{TW02}.
\end{proof}

\begin{lemma}\label{octagoncommutatorrelations}
	Using the notations of \cite{TW02} we obtain the following commutator relation:
	\begin{align*}
		[x_1(t), x_8(u_1, u_2)] = &x_2( t^{\sigma +1}u_1 + t^{\sigma+1} u_2^{\sigma+1}, tu_2 ) \\
		\cdot &x_3( t^{\sigma+1} u_2^{\sigma+2} + t^{\sigma+1} u_1 u_2 + t^{\sigma+1} u_1^{\sigma} ) \\
		\cdot &x_4( t^{\sigma+2} u_1^{\sigma} u_2^{\sigma+1} + t^{\sigma +2} u_1^2 u_2 + t^{\sigma +2} u_2^{2\sigma +3}, t^{\sigma} u_1 ) \\
		\cdot &x_5( t^{\sigma+1} u_1^{\sigma} u_2^{\sigma} + t^{\sigma+1} u_2^{2\sigma+2} + t^{\sigma+1} u_1^2 ) \\
		\cdot &x_6( t^{\sigma+1} u_1^2 u_2 + t^{\sigma+1} u_1^{\sigma} u_2^{\sigma+1} + t^{\sigma+1} u_1^{\sigma+1}, tu_1 + tu_2^{\sigma+1} ) \\
		\cdot &x_7( tu_1^{\sigma} + tu_1 u_2 + tu_2^{\sigma+2} )
	\end{align*}
\end{lemma}
\begin{proof}
	This is a straight forward computation, using only the commutator relations in Example $(16.9)$ of \cite{TW02}. For the computation see Lemma \ref{proofofoctagoncommutatorrelations} in the appendix.
\end{proof}

\subsection*{Affine RGD-systems}

\begin{proposition}\label{Propsimplylacedaffine}
	Every RGD-system of simply-laced affine type or of type $(2, 3, 6)$ satisfies Condition $\nc$.
\end{proposition}
\begin{proof}
	Our proof uses the fact that any affine Moufang building has a spherical Moufang building at infinity. Before we start the proof, we will describe this shortly: Let $\tilde{\mathcal{D}} = (\tilde{G}, (\tilde{U}_{\alpha})_{\alpha \in \tilde{\Phi}})$ be an RGD-system of affine type $\tilde{X}_{n-1}$. Then there exists an RGD-system $\mathcal{D} = (G, (U_{\alpha})_{\alpha \in \Phi})$ of spherical type $X_{n-1}$ such that $U_{\alpha}, U_{\beta} \leq U_{\gamma}$, where $\alpha, \beta \in \tilde{\Phi}$ such that $\alpha \subseteq \beta$ and some $\gamma \in \Phi$ (cf. \cite[Proposition $11.110$]{AB08}).
	
	Now we start the proof of the proposition. Let $(G, (U_{\alpha})_{\alpha \in \Phi})$ be an RGD-system of simply-laced affine type $(W, S)$ or of type $(2, 3, 6)$ and let $\alpha \subsetneq \beta \in \Phi$. Then $U_{\alpha}, U_{\beta}$ are subgroups of an abelian group by Lemma \ref{hexagonproperties}. Thus $[U_{\alpha}, U_{\beta}] = 1$.
\end{proof}

\begin{remark}\label{ExampleC2tilde}
	The following is independent of Condition $\costar$: 
	\begin{enumerate}[label=(\alph*)]
		\item There are RGD-systems of type $(2, 4, 4)$ with abelian root groups at infinity.
		
		\item There are RGD-systems of type $(2, 4, 4)$ with non-abelian root groups at infinity.
	\end{enumerate}
\end{remark}

\section{Commutator relations}

In this section we let $\mathcal{D} = (G, (U_{\alpha})_{\alpha \in \Phi})$ be an RGD-system of type $(W, S)$.

\begin{remark}
	Let $G$ be a group and let $A, B, C \leq G$ be three subgroups of $G$. If $[[A, B], C] = [[A, C], B] = 1$ then $[A, [B, C]]=1$. This is known in the literature as the \textit{three subgroup lemma} (see also $(2.3)$ of \cite{TW02}).
\end{remark}

\begin{lemma}\label{geometrictrianglecomplete}
	Assume that the Coxeter diagram is the complete graph and let $\{ \alpha, \beta, \gamma \}$ be a triangle. Then $[U_{-\alpha}, U_{\beta}] = [U_{-\alpha}, U_{\gamma}] = 1$.
\end{lemma}
\begin{proof}
	This follows directly from Proposition \ref{completefundamentaltriangle}.
\end{proof}

\begin{lemma}\label{geometrictrianglenotcomplete}
	Let $(W, S)$ be of type $(2, 6, 6), (2, 6, 8)$ or $(2, 8, 8)$ and let $\{ \alpha, \beta, \gamma \}$ be a triangle. Then we have $[[U_{-\alpha}, U_{\beta}], U_{\gamma}] = [[U_{-\alpha}, U_{\gamma}], U_{\beta}]=1$.
\end{lemma}
\begin{proof}
	The claim follows from Lemma \ref{angle1over2}, Lemma \ref{hexagonproperties} and Lemma \ref{octagonproperties}.
\end{proof}

\begin{remark}
	If $m_{st} \in \{4, 6\}$ there exist so called \textit{long roots} and \textit{short roots}. These are indicated by the \textit{Dynkin diagram}. An arrow from $s$ to $t$ means that the node $t$ corresponds to a short root.
\end{remark}

\begin{theorem}\label{Mainresult}
	Assume that $\mathcal{D}$ satisfies Condition $\costar$. Then $\mathcal{D}$ satisfies Condition $\nc$, if one of the following hold:
	\begin{enumerate}[label=(\alph*)]
		\item The Coxeter diagram of $(W, S)$ is the complete graph.
		
		\item $(W, S)$ is of type $(2, 3, 8)$ or of type $(2, 8, 8)$.
		
		\item $(W, S)$ is simply-laced.
		
		\item $(W, S)$ is of type $(2, 6, 6)$ and has Dynkin diagram \begin{tikzpicture}
			\node[dynkinnode] (A) at (0,0) {};
			\node[dynkinnode] (B) at (1,0) {};
			\node[dynkinnode] (C) at (2,0) {};
			
			\draw (A) -- (B) -- (C);
			\doubleline{B}{A};
			\doubleline{B}{C}
		\end{tikzpicture} or \begin{tikzpicture}
			\node[dynkinnode] (A) at (0,0) {};
			\node[dynkinnode] (B) at (1,0) {};
			\node[dynkinnode] (C) at (2,0) {};
			
			\draw (A) -- (B) -- (C);
			\doubleline{A}{B};
			\doubleline{C}{B}
		\end{tikzpicture}.
		
		\item $(W, S)$ is of type $(2, 6, 8)$ and has diagram \begin{tikzpicture}
		\node[dynkinnode] (A) at (0,0) {};
		\node[dynkinnode] (B) at (1,0) {};
		\node[dynkinnode] (C) at (2,0) {};
		
		\node[above] (bc) at (1.5, 0) {$8$};
		
		\draw (A) -- (B) -- (C);
		\doubleline{B}{A}
		\end{tikzpicture}
	\end{enumerate}
\end{theorem}
\begin{proof}
	Let $\alpha \subsetneq \beta$ be two roots. Then there exists a minimal gallery $(c_0, \ldots, c_k)$ such that $(\alpha_1 = \alpha, \ldots, \alpha_k = \beta)$ is the sequence of roots which are crossed by that minimal gallery, i.e. $\{ c_{i-1}, c_i \} \in \partial \alpha_i$. We can assume that $k$ is minimal with this property. We prove the hypothesis by induction on $k$. For $k \leq 2$ there is nothing to show. For $k=3$ we have $(\alpha, \beta) = \emptyset$ and hence $[U_{\alpha}, U_{\beta}] = 1$. Now let $k>3$ and let $R$ be the rank $2$ residue containing $c_{k-2}, c_{k-1}, c_k$. Then there exists a minimal gallery $(d_0 = c_0, \ldots, d_k = c_k)$ such that $d_j = \proj_R c_0$ for some $j$. Since the set of roots which are crossed by the gallery $(c_0, \ldots, c_k)$ equals the set of roots which are crossed by $(d_0, \ldots, d_k)$, the minimality of $k$ yields that $\alpha$ (resp. $\beta)$ is the first (resp. last) root which is crossed by the gallery $(d_0, \ldots, d_k)$. Thus we can assume that $\proj_R c_0$ is a chamber on the minimal gallery $(c_0, \ldots, c_k)$. Let $c_{j-1} = \proj_R c_0$. Then there exists a root $\gamma \in \Phi$ such that $\{ \gamma, \alpha_j \}$ is a root basis of the set of roots associated to $R$. Applying Lemma \ref{costarsimplerggenerate} we obtain $U_{\beta} \subseteq [U_{\alpha_j}, U_{\gamma}]$. If $o(r_{\alpha} r_{\alpha_j}) = o(r_{\alpha} r_{\gamma}) = \infty$, then $U_{\alpha}$ commutes with $U_{\alpha_j}$ and $U_{\gamma}$ by induction. Hence $U_{\alpha}$ commutes with the group generated by $U_{\alpha_j}$ and $U_{\gamma}$. Since this subgroup contains $U_{\beta}$ the claim follows. Thus we can assume that at least one of $o(r_{\alpha} r_j), o(r_{\alpha} r_{\gamma})$ is finite. In each case we will argue as before or use the three subgroup lemma. Thus we forget the root $\beta$ for the moment and focus on the roots $\alpha, \alpha_j, \gamma$. Thus the role of the roots $\alpha_j$ and $\gamma$ can be swapped. W.l.o.g. we can assume that $o(r_{\alpha} r_j) < \infty$. Then we consider the following two cases:
	\begin{enumerate}[label=(\alph*)]
		\item $o(r_{\alpha} r_{\gamma}) < \infty$: Then the set $\{ -\alpha, \alpha_j, \gamma \}$ is a triangle. If $(W, S)$ is as in $(a)$, we obtain $[U_{\alpha}, U_{\alpha_j}] = 1 = [U_{\alpha}, U_{\gamma}]$ by Lemma \ref{geometrictrianglecomplete}. As before, the claim follows. If $(W, S)$ is of type $(2, 6, 6), (2, 6, 8)$ or $(2, 8, 8)$ we use Lemma \ref{geometrictrianglenotcomplete} and the three subgroup lemma and the claim follows. If $(W, S)$ is of type $(2, 3, 8)$ we will show that $[[U_{\alpha}, U_{\alpha_j}], U_{\gamma}] = [[U_{\alpha}, U_{\gamma}], U_{\alpha_j}] = 1$. W.l.o.g. we can assume that there exists $\delta \in (\alpha, \alpha_j)$ such that $[U_{\delta}, U_{\gamma}] \neq 1$ (otherwise we are done). Furthermore, we have $\angle r_{\gamma} r_{\delta} \in \{ \frac{1}{3}, \frac{1}{2}, \frac{2}{3} \}$. Using Lemma \ref{octagonproperties}, we obtain $m_{\delta, \gamma} = 3$. Applying Proposition \ref{Prop238} we obtain $n := \vert (\alpha, \alpha_j) \vert \leq 3$ and $(\alpha, \gamma) = \emptyset$. This implies $[U_{\alpha}, U_{\gamma}] = 1$. If $n=1$ then the triangle $\{ -\alpha, \delta, \gamma \}$ is fundamental and hence $\angle r_{\gamma} r_{\alpha} = \frac{1}{2}$. Thus we have $o(r_{\alpha} r_{\epsilon}) < \infty$ and $(\alpha, \epsilon) = \emptyset$, where $\{ \epsilon \} = (\delta, \gamma)$. Since $\angle r_j r_{\epsilon} = \frac{1}{2}$, we obtain $[U_{\alpha_j}, U_{\epsilon}] = 1$ and hence $[U_{\alpha}, U_{\beta}] \subseteq U_{\epsilon} \cap U_{(\alpha, \beta)} = 1$. This finishes the claim. If $n=2$ then $U_{\alpha}$ corresponds to $U_{2k+1}$ of Example $(16.9)$ in \cite{TW02}, since $U_{\delta}$ is abelian. Using the commutator relations and the fact that $(\alpha, \delta) = \{ \delta_0 \}$ and $\angle r_{\gamma} r_{\delta_0} \in \{ \frac{1}{3}, \frac{1}{2} \}$, we obtain $[[U_{\alpha}, U_{\alpha_j}], U_{\gamma}] = 1$. If $n=3$, then we have $\angle r_{\alpha} r_j = \frac{1}{2}$ and we obtain $[U_{\alpha}, U_{\alpha_j}] = 1$ by Lemma \ref{octagonproperties}. Now we assume that $(W, S)$ is simply-laced. Then we obtain that $T := \{ r_{\alpha}, r_j, r_{\gamma} \}$ is an affine reflection triangle. Using Lemma \ref{Theorem1.2CM} we obtain an irreducible affine parabolic subgroup $W_0 \leq W$ of rank at least $3$ such that $\langle T \rangle$ is conjugated to a subgroup of $W_0$. Let $g\in W$ such that $\langle T \rangle^g \leq W_0$. Then $r_{\alpha}^g, r_{\beta}^g \in W_0$. In view of Proposition \ref{Propsimplylacedaffine} the residue corresponding to $W_0$ satisfies Condition $\nc$ and we have $[U_{r_{\alpha}^g}, U_{r_{\beta}^g}] = 1$. Using $r_{\alpha}^g = r_{g^{-1} \alpha}$, we obtain
		\[ [U_{\alpha}, U_{\beta}]^g = [U_{\alpha}^g, U_{\beta}^g] = [U_{g^{-1}\alpha}, U_{g^{-1}\beta}] = 1 \]
		
		\item $o(r_{\alpha} r_{\gamma}) = \infty$: Then we have $\alpha \subsetneq \gamma$. Using induction we obtain $[U_{\alpha}, U_{\gamma}] = 1$. If $o(r_{\delta} r_{\gamma}) = \infty$ for any $\delta \in (\alpha, \alpha_j)$ then the claim follows by induction and the three subgroup lemma. Thus we can assume that there exists $\delta \in (\alpha, \alpha_j)$ such that $\{ -\delta, \alpha_j, \gamma \}$ is a triangle. We first assume that $(W, S)$ is as in $(a)$. By Lemma \ref{geometrictrianglecomplete} we have $[U_{\delta}, U_{\gamma}] = 1$. The claim follows now from the three subgroup lemma. Let $(W, S)$ be of type $(2, 8, 8)$ and assume that $(\delta, \gamma) \neq \emptyset$. Then we have $\vert (\delta, \gamma) \vert = 1, (\delta, \alpha_j) = \emptyset$ and we have $(\delta', \gamma) = \emptyset$ for any $\delta' \in (\alpha, \delta)$ such that $o(r_{\delta'} r_{\gamma}) < \infty$ by Lemma \ref{angle1over2}. Since $(W, S)$ is of type $(2, 8, 8)$, we have $m_{\alpha, \alpha_j} = m_{\delta, \gamma} = 8$. We assume that $U_{\alpha_j}$ corresponds to $x_{2k}$ in Example $(16.9)$ of \cite{TW02}. Then $U_{\delta}$ corresponds to $x_{2k-1}$ and hence $[U_{\delta}, U_{\gamma}] =1$ by Lemma \ref{octagonproperties}. If $U_{\alpha_j}$ corresponds to $x_{2k-1}$ we distinguish the following two cases:
		\begin{enumerate}[label=(\roman*)]
			\item $\vert (\alpha, \alpha_j) \vert < 6$: It follows from the commutator relations of \cite{TW02}, that the $U_{\delta}$-part of $[U_{\alpha}, U_{\beta}]$ is contained in the centralizer of $U_{\delta}$ and such an element commutes with $U_{\gamma}$. Using induction and the three subgroup lemma the claim follows.

			\item $\vert (\alpha, \alpha_j) \vert = 6$: Because of the angles it follows that $o(r_{\alpha} r_{\epsilon}) < \infty$, where $\{ \epsilon \} = (\delta, \gamma)$. Using induction and the three subgroup lemma, we obtain $[U_{\alpha}, U_{\beta}] \subseteq U_{\epsilon} \cap U_{(\alpha, \beta)} = 1$ as above.
		\end{enumerate}
	
		Now we assume that $(W, S)$ is of type $(2, 6, 6)$ and the Dynkin diagram is as in the statement. We use the notations of \cite{TW02}. Then the middle node is in both residues either the field or the vector space. We can assume that $(\delta, \gamma) \neq \emptyset$. We obtain $m_{\alpha, \beta} = 6$ and $\vert (\delta, \gamma) \vert = 1$ by Lemma \ref{angle1over2}. We assume that $U_{\alpha_j}$ corresponds to $x_{2k-1}$ in Example $(16.8)$ of \cite{TW02}. Then $[U_{\delta}, U_{\gamma}]=1$. Now we assume that $U_{\alpha_j}$ corresponds to $x_{2k}$. If $k= 2$ we have $[U_{\alpha}, U_{\alpha_j}] = 1$. Thus we assume $k \in \{1, 3\}$. If $U_{\alpha}$ does not correspond to $U_1$ in Example $(16.9)$ of \cite{TW02} we obtain $[U_{\alpha}, U_{\alpha_j}] \subseteq \langle U_{\delta'} \mid \delta' \in (\alpha, \alpha_j) \backslash \{ \delta \} \rangle$ and $(\delta', \gamma) = \emptyset$ for any $\delta \neq \delta' \in (\alpha, \alpha_j)$ with $o(r_{\delta'} r_{\gamma}) < \infty$. Otherwise we obtain $o(r_{\alpha} r_{\delta'}) <\infty$, where $\delta' \in (\delta, \gamma)$ and $[U_{\delta'}, U_{\alpha_j}] = 1 = [U_{\delta'}, U_{\gamma}]$. Now the claim follows from the three subgroup lemma.
		
		Let $(W, S)$ be of type $(2, 6, 8)$ with diagram as in the statement. Again we use the notations of \cite{TW02}. We can assume that there exists $\epsilon \in (\delta, \gamma)$. Using Lemma \ref{angle1over2} we obtain that $\{ -\delta, \beta, \epsilon \}$ and $\{ -\epsilon, \gamma, \beta \}$ are fundamental triangles. Thus $U_{\delta}$ must be parametrized by the middle node. Using the commutator relations we obtain in both cases ($m_{\delta, \gamma} = 8$ and $m_{\delta, \gamma} = 6$) that $[U_{\delta}, U_{\gamma}] = 1$. Furthermore, we have $(\delta', \gamma) = \emptyset$ for any $\delta \neq \delta' \in (\alpha, \beta)$ with $o(r_{\delta'} r_{\gamma}) < \infty$. Using induction and the three subgroup lemma the claim follows.
		
		Let $(W, S)$ be of type $(2, 3, 8)$ and assume that $(\delta, \gamma) \neq \emptyset$. We distinguish the following cases: If $m_{\alpha, \alpha_j} = 3$, then $\vert (\delta, \gamma) \vert \in \{1, 2, 3\}$ and $m_{\delta, \gamma} = 8$. Since $U_{\delta}$ is abelian we have $[U_{\delta}, U_{\gamma}]=1$ for $\vert (\delta, \gamma) \vert \in \{1, 3\}$ by Lemma \ref{octagonproperties}. Now we assume that $\vert (\delta, \gamma) \vert = 2$. Let $\{\delta', \delta''\} = (\delta, \gamma)$ and assume that $\angle r_{\gamma} r_{\delta''} = \angle r_{\delta''} r_{\delta'} = \angle r_{\delta'} r_{\delta} = \frac{1}{8}$. Using the commutator relations of $(16.9)$ of \cite{TW02} one obtains that $[U_{\delta}, U_{\gamma}] \subseteq U_{\delta'}$ and $[[U_{\delta}, U_{\gamma}], U_{\gamma}] = 1$. Using the hyperbolic plane one obtains that $o(r_{\alpha} r_{\delta'}) < \infty$ and hence $[U_{\alpha}, U_{\beta}] \subseteq U_{\delta'} \cap U_{(\alpha, \beta)} = 1$ as above. Now we assume that $m_{\alpha, \alpha_j} = 8$. If $m_{\delta, \gamma} = 8$ one can show that $(\delta, \alpha_j) = \emptyset$ and the claim follows by Lemma \ref{Prop238auxres} and Lemma \ref{octagonproperties}. Thus we can assume that $m_{\delta, \gamma} = 3$. This implies $m_{\alpha, \alpha_j} = 8$ and $(\delta, \alpha_j) = \emptyset$. If $n := \vert (\alpha, \alpha_j) \vert \geq 5$ then $o(r_{\alpha} r_{\delta'}) < \infty$, where $\{ \delta' \} = (\delta, \gamma)$ and the claim follows as above using Lemma \ref{Prop238auxres} and induction. If $n \leq 3$ we obtain $o(r_{\alpha} r_{\gamma}) < \infty$ which is a contradiction to the assumption. Thus we can assume $n=4$. Since $[U_{2i+1}, U_{2i+6}]_{2i+5} = 1$ the claim follows from Lemma \ref{Prop238auxres}, induction and the three subgroup lemma. If $(W, S)$ is simply-laced, then $\{ -\delta, \alpha_j, \gamma \}$ is a triangle. Using similar arguments as in the previous case the claim follows. \qedhere
	\end{enumerate}
\end{proof}

\section{Non-cyclic hyperbolic cases}\label{sec:Non-cyclic}

As we have seen in Theorem \ref{Mainresult}, Condition $\costar$ implies Condition $\nc$ in a lot of cases, where $(W, S)$ is of rank $3$ and of hyperbolic type. In this section we will consider the hyperbolic cases of rank $3$ which were not mentioned in Theorem \ref{Mainresult}. In this section we let $\mathcal{D} = (G, (U_{\alpha})_{\alpha \in \Phi})$ be an RGD-system of type $(W, S)$. In the proofs of this section we use the notations of \cite{TW02}. From now on by a Coxeter complex we mean the geometric realization of $\Sigma(W, S)$.

\subsection*{The case $(2, 4, 6)$}

\begin{theorem}
	There exists an RGD-system of type $(2, 4, 6)$, in which Condition $\costar \ldots$
	\begin{enumerate}[label=(\alph*)]
		\item $\ldots$ implies Condition $\nc$.
		
		\item $\ldots$ does not imply Condition $\nc$.
	\end{enumerate}
\end{theorem}
\begin{proof}
	Let $\KK$ be a field. We take the hexagonal system $(\mathbb{E} / \KK)^{\circ}$ of type $(1/\KK)$ (cf. Example $(15.20)$ of \cite{TW02}) and for the quadrangle we take $\mathcal{Q}_I(\KK, \KK_0, \sigma)$ (cf. Example $(16.2)$ of \cite{TW02}). We assume that the middle node of the diagram is parametrized by the field $\KK$:
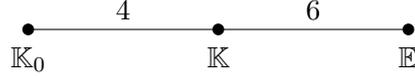
	\[ \begin{tikzpicture}
	
	\node[below] (a) at (0, -0.1) {$\KK_0$};
	\node[below] (b) at (2, -0.1) {$\KK$};
	\node[below] (c) at (4, -0.1) {$\mathbb{E}$};
	
	\node[above] (d) at (1, 0) {$4$};
	\node[above] (e) at (3, 0) {$6$};
	
	\filldraw [black] 	(0, 0) circle (2pt)
	(2, 0) circle (2pt)
	(4, 0) circle (2pt);
	\draw (0, 0) -- (2, 0) -- (4, 0);
\end{tikzpicture} \]

In the Coxeter complex of type $(2, 4, 6)$ there exists the following figure such that $(\epsilon', \delta) = (\epsilon', \gamma) = (\alpha, \gamma) = \emptyset, \angle r_{\epsilon '} r_{\epsilon} = \frac{1}{4}$ and $\angle r_{\alpha} r_{\epsilon} = \angle r_{\epsilon} r_{\delta} = \angle r_{\delta} r_{\gamma} = \frac{1}{6}$:
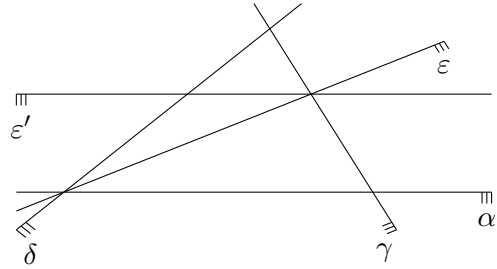
\[ \begin{tikzpicture}[scale=0.5]

\draw (0  ,  0  ) to (10 ,  0  );
\draw (10 ,  0  ) to (10 , -0.3);
\draw (9.9,  0  ) to (9.9, -0.3);
\draw (9.8,  0  ) to (9.8, -0.3);
\draw (9.9, -0.3) node [below]{$\alpha$};

\draw (0, -1) to (6, 5);
\draw (0  , -1  ) to (0.2, -1.2);
\draw (0.1, -0.9) to (0.3, -1.1);
\draw (0.2, -0.8) to (0.4, -1  );
\draw (0.3, -1.1) node [below]{$\delta$};

\draw (5, 5) to (8, -1);
\draw (8  , -1  ) to (7.8, -1.1);
\draw (7.95, -0.9) to (7.75, -1);
\draw (7.9, -0.8) to (7.7, -0.9);
\draw (7.75, -1.1) node [below]{$\gamma$};

\draw (0  , -0.5 ) to (9, 4);
\draw (9  ,  4   ) to (9.1, 3.8);
\draw (8.9,  3.95) to (9  , 3.75);
\draw (8.8,  3.9 ) to (8.9, 3.7);
\draw (9, 3.75) node [below]{$\epsilon$};

\draw (0  , 2.6) to (10 , 2.6);
\draw (0  , 2.6) to (0  , 2.3);
\draw (0.1, 2.6) to (0.1, 2.3);
\draw (0.2, 2.6) to (0.2, 2.3);
\draw (0.1, 2.3) node [below]{$\epsilon'$};

\end{tikzpicture} \]

Thus $U_{\epsilon}$ is parametrized by the additive group of the field $\KK$. Using induction and the commutator relations (cf. Chapter $16$ of \cite{TW02}) we obtain:
\[ \left[ x_{\alpha}(v), \prod_{i=1}^{n} [x_{\delta}(w_i), x_{\gamma}(k_i)] \right] = x_{\epsilon'}\left( -\sum_{i=1}^{n} T(v, w_i)^{\sigma}k_i + k_i^{\sigma} T(v, w_i) \right) \]

Now let $\sigma = \id$ and $\mathbb{E} = \KK = \KK_0$. Using the definition of $T$ and $\sigma$ in this example the commutator relation above reduces to
\[ \left[ x_{\alpha}(v), \prod_{i=1}^{n} [x_{\delta}(w_i), x_{\gamma}(k_i)] \right] = x_{\epsilon'}\left( -6v \sum_{i=1}^{n} w_i k_i \right) \]	

Now we assume that $\KK$ is a field of characteristic different from $2$ and $3$. Then we obtain
\[ [x_{\alpha}(1), [x_{\delta}(1), x_{\gamma}(1)]] = x_{\epsilon'}(-6) \neq 1 \]
Thus there must be a root in the residue with root basis $\{ \delta, \gamma \}$, which does not commute with $x_{\alpha}$ (otherwise the previous commutator would be trivial). Hence this RGD-system does not satisfy Condition $\nc$ and part $(b)$ follows. For part $(a)$ we could take a field of characteristic $2$ with at least $4$ elements. The existence of such an RGD-system follows from the theory of groups of Kac-Moody type. For example one can take the split Kac-Moody group over $\mathbb{K}$ of type $(W, S)$ with Cartan matrix $\begin{pmatrix}
2 & -1 & 0 \\
-2 & 2 & -1 \\
0 & -3 & 2
\end{pmatrix}$.
\end{proof}

\subsection*{The case $(2, 6, 6)$}

\begin{theorem}
	There exists an RGD-system of type $(2, 6, 6)$, in which Condition $\costar$ implies Condition $\nc$.
\end{theorem}
\begin{proof}
	Consider the split Kac-Moody group over $\mathbb{K} \neq \{ \mathbb{F}_2, \FF_3 \}$ of type $(2, 6, 6)$ with Cartan matrix $\begin{pmatrix}
	2 & -3 & 0 \\ -1 & 2 & -1 \\ 0 & -3 & 2
	\end{pmatrix}$ or $\begin{pmatrix}
	2 & -1 & 0 \\ -3 & 2 & -3 \\ 0 & -1 & 2
	\end{pmatrix}$. Then the corresponding Dynkin diagram is as in Theorem \ref{Mainresult} and the claim follows.
\end{proof}

\begin{theorem}
	There exists an RGD-system of type $(2, 6, 6)$, in which Condition $\costar$ does not imply Condition $\nc$.
\end{theorem}
\begin{proof}
	Assume that \begin{tikzpicture}
	\node[dynkinnode] (A) at (0,0) {};
	\node[dynkinnode] (B) at (1,0) {};
	\node[dynkinnode] (C) at (2,0) {};
	
	\draw (A) -- (B) -- (C);
	\doubleline{A}{B};
	\doubleline{B}{C}
	\end{tikzpicture} is the Dynkin diagram of the RGD-system. Let $\KK$ be a field and $V$ be a vector space over $\KK$. Then the right node is parametrized by $V$ and the other two nodes are parametrized by $\KK$.

	In the Coxeter complex of type $(2, 6, 6)$ there exists the following figure such that $(\epsilon', \delta) = (\epsilon', \gamma) = (\alpha, \gamma) = \emptyset, \angle r_{\epsilon '} r_{\epsilon} = \angle r_{\alpha} r_{\epsilon} = \angle r_{\epsilon} r_{\delta} = \angle r_{\delta} r_{\gamma} = \frac{1}{6}$ and $o(r_{\alpha} r_{\epsilon'}) = o(r_{\alpha} r_{\gamma}) = \infty$:
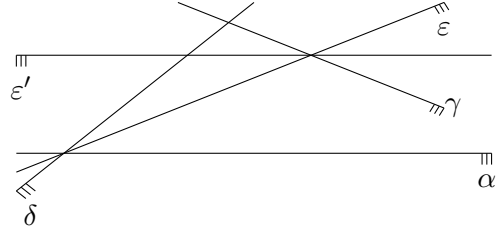
\[ \begin{tikzpicture}[scale=0.5]

\draw (0  ,  0  ) to (10 ,  0  );
\draw (10 ,  0  ) to (10 , -0.3);
\draw (9.9,  0  ) to (9.9, -0.3);
\draw (9.8,  0  ) to (9.8, -0.3);
\draw (9.9, -0.3) node [below]{$\alpha$};

\draw (0, -1) to (5, 4);
\draw (0  , -1  ) to (0.2, -1.2);
\draw (0.1, -0.9) to (0.3, -1.1);
\draw (0.2, -0.8) to (0.4, -1  );
\draw (0.3, -1.1) node [below]{$\delta$};

\draw (3.4, 4) to (9, 1.2);
\draw (9, 1.2) to (8.9, 1.0);
\draw (8.9, 1.25) to (8.8, 1.05);
\draw (8.8, 1.3) to (8.7, 1.1);
\draw (8.8, 1.2) node [right]{$\gamma$};

\draw (0  , -0.5 ) to (9, 4);
\draw (9  ,  4   ) to (9.1, 3.8);
\draw (8.9,  3.95) to (9  , 3.75);
\draw (8.8,  3.9 ) to (8.9, 3.7);
\draw (9, 3.75) node [below]{$\epsilon$};

\draw (0  , 2.6) to (10 , 2.6);
\draw (0  , 2.6) to (0  , 2.3);
\draw (0.1, 2.6) to (0.1, 2.3);
\draw (0.2, 2.6) to (0.2, 2.3);
\draw (0.1, 2.3) node [below]{$\epsilon'$};

\end{tikzpicture} \]

We can assume that $U_{\alpha}$ is parametrized by the vector space $V$. Then $U_{\epsilon}$ is parametrized by the field $\KK$. We remark that $[U_{\epsilon}, U_{\gamma}] \neq 1$ because of the diagram. We use for the hexagon the hexagonal system defined in $(15.20)$ of \cite{TW02}. Using the commutator relations (cf. Chapter $16$ of \cite{TW02}) we obtain:
\[ \left[ x_{\alpha}(v), \prod_{i=1}^{n} [x_{\delta}(w_i), x_{\gamma}(k_i)] \right] = x_{\epsilon'}\left( \sum_{i=1}^{n} T\left( T(v, w_i), k_i \right) \right) = x_{\epsilon '}\left( 9v \sum_{i=1}^{n} w_i k_i \right) \]
For $0 \neq w\in V$ we obtain $[x_{\alpha}(1), [x_{\delta}(w), x_{\gamma}(1)]] = x_{\epsilon'}(9w) \neq 1$ if the characteristic of $\KK$ is different from $3$ and hence such an RGD-system does not satisfy Condition $\nc$. An example of such an RGD-system is provided by the split Kac-Moody group over $\mathbb{K} \neq \mathbb{F}_2$ of characteristic $\neq 3$ with Cartan matrix $\begin{pmatrix}
2 & -1 & 0 \\ -3 & 2 & -1 \\ 0 & -3 & 2
\end{pmatrix}$.
\end{proof}

\subsection*{The case $(2, 4, 8)$}

\begin{theorem}
	There exists an RGD-system of type $(2, 4, 8)$ such that Condition $\costar \ldots$
	\begin{enumerate}[label=(\alph*)]
		\item $\ldots$ implies Condition $\nc$.

		\item $\ldots$ does not imply Condition $\nc$.
	\end{enumerate}
\end{theorem}
\begin{proof}
	Let $(\KK, \sigma)$ be as in Example $(10.12)$ of \cite{TW02}. Let $\mathbb{E}$ be a field containing $\KK$ and let $a, b \in \mathbb{E}$ algebraically independent over $\KK$ such that $\mathbb{E} = \KK(a, b)$. We extend $\sigma$ to an endomorphism of $\mathbb{E}$ by setting $a^{\sigma} = b$ and $b^{\sigma} = a^2$. Then $(\mathbb{E}, \sigma)$ is an octagonal set as in Example $(10.12)$ of \cite{TW02}. For the quadrangle we take $\mathcal{Q}_I(\mathbb{E}, \mathbb{E}_0, \tau)$ (cf. Example $(16.2)$ of \cite{TW02}), where $\tau$ fixes $\KK$ pointwise and interchanges $a$ and $b$. We assume that the middle node of the diagram is parametrized by the field $\mathbb{E}$:
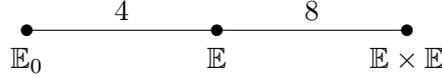
\[ \begin{tikzpicture}

\node[below] (a) at (0, -0.1) {$\mathbb{E}_0$};
\node[below] (b) at (2, -0.1) {$\mathbb{E}$};
\node[below] (c) at (4, -0.1) {$\mathbb{E}\times \mathbb{E}$};

\node[above] (d) at (1, 0) {$4$};
\node[above] (e) at (3, 0) {$8$};

\filldraw [black] 	(0, 0) circle (2pt)
(2, 0) circle (2pt)
(4, 0) circle (2pt);
\draw (0, 0) -- (2, 0) -- (4, 0);
\end{tikzpicture} \]

In the Coxeter complex of type $(2, 4, 8)$ there exists the following figure such that $(\epsilon', \delta) = (\epsilon', \gamma) = (\alpha, \gamma) = \emptyset, m_{\epsilon, \gamma} = 4$ and $\angle r_{\alpha} r_{\epsilon} = \angle r_{\epsilon} r_{\delta} = \angle r_{\delta} r_{\gamma} = \frac{1}{8}$:
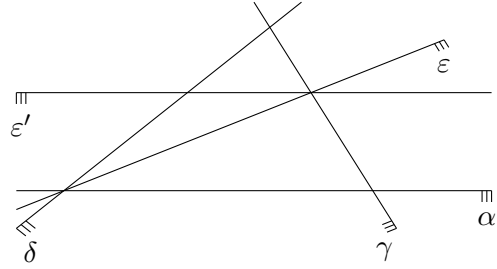
\[ \begin{tikzpicture}[scale=0.5]

\draw (0  ,  0  ) to (10 ,  0  );
\draw (10 ,  0  ) to (10 , -0.3);
\draw (9.9,  0  ) to (9.9, -0.3);
\draw (9.8,  0  ) to (9.8, -0.3);
\draw (9.9, -0.3) node [below]{$\alpha$};

\draw (0, -1) to (6, 5);
\draw (0  , -1  ) to (0.2, -1.2);
\draw (0.1, -0.9) to (0.3, -1.1);
\draw (0.2, -0.8) to (0.4, -1  );
\draw (0.3, -1.1) node [below]{$\delta$};

\draw (5, 5) to (8, -1);
\draw (8  , -1  ) to (7.8, -1.1);
\draw (7.95, -0.9) to (7.75, -1);
\draw (7.9, -0.8) to (7.7, -0.9);
\draw (7.75, -1.1) node [below]{$\gamma$};

\draw (0  , -0.5 ) to (9, 4);
\draw (9  ,  4   ) to (9.1, 3.8);
\draw (8.9,  3.95) to (9  , 3.75);
\draw (8.8,  3.9 ) to (8.9, 3.7);
\draw (9, 3.75) node [below]{$\epsilon$};

\draw (0  , 2.6) to (10 , 2.6);
\draw (0  , 2.6) to (0  , 2.3);
\draw (0.1, 2.6) to (0.1, 2.3);
\draw (0.2, 2.6) to (0.2, 2.3);
\draw (0.1, 2.3) node [below]{$\epsilon'$};

\end{tikzpicture} \]

Thus $U_{\epsilon}$ is parametrized by the additive group of the field $\mathbb{E}$. Using induction and the commutator relations (cf. Chapter $16$ of \cite{TW02}) we obtain:
\[ \left[ x_{\alpha}(t, u), \prod_{i=1}^{n} [x_{\gamma}(k_i), x_{\delta}(t_i, u_i)] \right] = x_{\epsilon'}\left( \sum_{i=1}^{n} (uu_i)^{\tau}k_i + k_i^{\tau}uu_i \right) \]

This implies that $[x_{\alpha}(0, 1), [ x_{\gamma}(1), x_{\delta}( 1, a ) ] ] = x_{\epsilon'}( b + a ) \neq 1$ and hence there exists $2 \leq i \leq 7$ with $[x_{\alpha}(0, 1), x_i(k)] \neq 1$ for some $k \in \KK \cup \left( \KK \times \KK \right)$ (depending if $i$ is even or odd). This shows part $(b)$. For part $(a)$ we take the same example as above, except that the map $\tau = \id$ in the quadrangle. The existence of such RGD-systems is established by applying a Suzuki-Ree twist to a suitable split Kac-Moody group. The construction uses the same strategy as the construction given in \cite{He90}.
\end{proof}

\subsection*{The case $(2, 6, 8)$}

\begin{theorem}
	There exists an RGD-system of type $(W, S)$, in which Condition $\costar$ implies Condition $\nc$.
\end{theorem}
\begin{proof}
	An RGD-system with diagram \begin{tikzpicture}
	\node[dynkinnode] (A) at (0,0) {};
	\node[dynkinnode] (B) at (1,0) {};
	\node[dynkinnode] (C) at (2,0) {};
	
	\node[above] (bc) at (1.5, 0) {$8$};
	
	\draw (A) -- (B) -- (C);
	\doubleline{B}{A}
	\end{tikzpicture} is an example. The existence of such an RGD-system can be proved along the lines in \cite{He90}.
\end{proof}

\begin{theorem}
	There exists an RGD-system of type $(2, 6, 8)$, in which Condition $\costar$ does not imply Condition $\nc$.
\end{theorem}
\begin{proof}
	Let $(\KK, \sigma)$ be as in Example $(10.12)$ of \cite{TW02}. Let $\mathbb{E}$ be a field containing $\KK$ and let $a, b \in \mathbb{E}$ algebraically independent over $\KK$ such that $\mathbb{E} = \KK(a, b)$. We extend $\sigma$ to an endomorphism of $\mathbb{E}$ by setting $a^{\sigma} = b$ and $b^{\sigma} = a^2$. Then $(\mathbb{E}, \sigma)$ is an octagonal set as in Example $(10.12)$ of \cite{TW02}. For the hexagon we take the hexagonal system $(\mathbb{E}/\mathbb{E})^{\circ}$ of type $(1/\mathbb{E})$ (cf. Example $(15.20)$). We assume that \begin{tikzpicture}
	\node[dynkinnode] (A) at (0,0) {};
	\node[dynkinnode] (B) at (1,0) {};
	\node[dynkinnode] (C) at (2,0) {};
	
	\node[above] (bc) at (1.5, 0) {$8$};
	
	\draw (A) -- (B) -- (C);
	\doubleline{A}{B}
	\end{tikzpicture} is the diagram of the RGD-system. This means that the right node is parametrized by $\mathbb{E} \times \mathbb{E}$ and the other two nodes are parametrized by $\mathbb{E}$.

In the Coxeter complex of type $(2, 6, 8)$ there exists the following figure such that $(\epsilon', \delta) = (\epsilon', \gamma) = (\alpha, \gamma) = \emptyset, m_{\epsilon, \gamma} = 6, \angle r_{\alpha} r_{\epsilon} = \angle r_{\epsilon} r_{\delta} = \angle r_{\delta} r_{\gamma} = \frac{1}{8}$ and $o(r_{\alpha} r_{\gamma}) = o(r_{\alpha} r_{\epsilon'}) = \infty$:
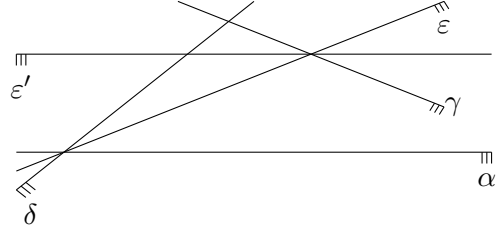
\[ \begin{tikzpicture}[scale=0.5]

\draw (0  ,  0  ) to (10 ,  0  );
\draw (10 ,  0  ) to (10 , -0.3);
\draw (9.9,  0  ) to (9.9, -0.3);
\draw (9.8,  0  ) to (9.8, -0.3);
\draw (9.9, -0.3) node [below]{$\alpha$};

\draw (0, -1) to (5, 4);
\draw (0  , -1  ) to (0.2, -1.2);
\draw (0.1, -0.9) to (0.3, -1.1);
\draw (0.2, -0.8) to (0.4, -1  );
\draw (0.3, -1.1) node [below]{$\delta$};

\draw (3.4, 4) to (9, 1.2);
\draw (9, 1.2) to (8.9, 1.0);
\draw (8.9, 1.25) to (8.8, 1.05);
\draw (8.8, 1.3) to (8.7, 1.1);
\draw (8.8, 1.2) node [right]{$\gamma$};

\draw (0  , -0.5 ) to (9, 4);
\draw (9  ,  4   ) to (9.1, 3.8);
\draw (8.9,  3.95) to (9  , 3.75);
\draw (8.8,  3.9 ) to (8.9, 3.7);
\draw (9, 3.75) node [below]{$\epsilon$};

\draw (0  , 2.6) to (10 , 2.6);
\draw (0  , 2.6) to (0  , 2.3);
\draw (0.1, 2.6) to (0.1, 2.3);
\draw (0.2, 2.6) to (0.2, 2.3);
\draw (0.1, 2.3) node [below]{$\epsilon'$};

\end{tikzpicture} \]

Thus $U_{\epsilon}$ is parametrized by the additive group of the field $\mathbb{E}$. Using induction and the commutator relations (cf. Chapter $16$ of \cite{TW02}) we obtain:
\[ \left[ x_{\alpha}(t, u), \prod_{i=1}^{n} [x_{\gamma}(k_i), x_{\delta}(t_i, u_i)] \right] = x_{\epsilon'}\left( \sum_{i=1}^{n} T(uu_i, k_i) \right) = x_{\epsilon'}\left( u\sum_{i=1}^{n} u_ik_i \right) \]

This implies that $[x_{\alpha}(0, 1), [ x_{\gamma}(1), x_{\delta}( 0, 1 ) ]] = x_{\epsilon'}( 1 ) \neq 1$. This shows that such an RGD-system does not satisfy Condition $\nc$. The existence of such RGD-systems is established by applying a Suzuki-Ree twist to a suitable split Kac-Moody group. The construction uses the same strategy as the construction given in \cite{He90}.
\end{proof}

\appendix
\renewcommand{\thesection}{Appendix \Alph{section}}

\renewcommand\thecountercheck{(\Alph{section}.\arabic{countercheck})}

\section{Computations in an RGD-system of type $G_2$}

Let $\mathcal{D} = (G, (U_{\alpha})_{\alpha \in \Phi})$ be an RGD-system of type $G_2$. If $\mathcal{D}$ satisfies Condition $\costar$ we can write every element of $U_{\gamma}$ as a product of elements $[u_{\alpha}, u_{\beta}]$, where $u_{\alpha} \in U_{\alpha}, u_{\beta} \in U_{\beta}, \Pi = \{ \alpha, \beta \}$ and $\gamma \in (\alpha, \beta)$. In this section we will state concrete elements $u_{\alpha}, u_{\beta}$ with the required property for some $u_{\gamma} \in U_{\gamma}$. We remark that we do not assume Condition $\costar$ in this section.

\begin{lemma}
	Using the notations of \cite{TW02} we obtain the following commutator relations:
	\begin{align*}
	[x_1(a_1), x_6(t_1)] [x_1(a_2), x_6(t_2)] = &x_2(-t_1N(a_1)-t_2N(a_2)) \\
	\cdot &x_3(t_1a_1^{\#} + t_2a_2^{\#}) \\
	\cdot &x_4(t_1^2N(a_1) + t_2^2N(a_2) + T(t_1a_1, t_2a_2^{\#})) \\
	\cdot &x_5(-t_1a_1 - t_2a_2)
	\end{align*}
\end{lemma}
\begin{proof}
	This is an elementary computation.
\end{proof}

\begin{remark}\label{hexagoncommutatorrelations}
	Using the notations of \cite{TW02} and the previous lemma we obtain the following commutator relations:
	\begin{enumerate}[label=(\alph*)]
		\item $[x_1(-k \cdot 1_V), x_6(k^{-1})] [x_1(-k \cdot 1_V), x_6(-k^{-1})] = x_4(k)$.
		
		\item $[x_1(v), x_6(k)] [x_1(-v), x_6(k)] = x_3(2kv^{\#}) x_4(3k^2 N(v))$.
		
		\item $[x_1(v), x_6(k^3)] [x_1(kv), x_6(-1)] = x_3(\left( k^3 - k^2 \right)v^{\#}) x_4(N(v) \left( k^6 + k^3 - 3k^5 \right)) x_5((k-k^3)v)$.
	\end{enumerate}
	If the characteristic of the field is different from $2$ and $v\neq 0$, we can choose $v' := v^{\#} \neq 0$ and $k' := (2N(v^{\#}))^{-1}$ and obtain the following:
	\[ [x_1(v'), x_6(k')] [x_1(-v'), x_6(k')] = x_3(v) x_4(3(k')^2 N(v')) \]
\end{remark}

\section{Computations in an RGD-system of type $I_2(8)$}

In this section we let $\mathcal{D} = (G, (U_{\alpha})_{\alpha \in \Phi})$ be an RGD-system of type $I_2(8)$.

\begin{lemma}\label{proofofoctagoncommutatorrelations}
	Using the notations of \cite{TW02} we obtain the following commutator relations:
	\begin{enumerate}[label=(\alph*)]
		\item The general commutator of two basis roots is given by
		\begin{align*}
			\everyafterautobreak{\cdot}
			\begin{autobreak}
			[x_1(t), x_8(u_1, u_2)] = 
			x_2( t^{\sigma +1}u_1 + t^{\sigma+1} u_2^{\sigma+1}, tu_2 ) 
			x_3( t^{\sigma+1} u_2^{\sigma+2} + t^{\sigma+1} u_1 u_2 + t^{\sigma+1} u_1^{\sigma} ) 
			x_4( t^{\sigma+2} u_1^{\sigma} u_2^{\sigma+1} + t^{\sigma +2} u_1^2 u_2 + t^{\sigma +2} u_2^{2\sigma +3}, t^{\sigma} u_1 ) 
			x_5( t^{\sigma+1} u_1^{\sigma} u_2^{\sigma} + t^{\sigma+1} u_2^{2\sigma+2} + t^{\sigma+1} u_1^2 ) 
			x_6( t^{\sigma+1} u_1^2 u_2 + t^{\sigma+1} u_1^{\sigma} u_2^{\sigma+1} + t^{\sigma+1} u_1^{\sigma+1}, tu_1 + tu_2^{\sigma+1} ) 
			x_7( tu_1^{\sigma} + tu_1 u_2 + tu_2^{\sigma+2} )
			\end{autobreak}
		\end{align*}
				
		\item The multiplication of two general commutators of basis roots is given by
		\begin{align*}
			\begin{autobreak}
			[x_1(t_1), x_8(u_1, u_2)] 
			[x_1(t_2), x_8(v_1, v_2)] 
			= x_2( t_1^{\sigma+1} \left( u_1 + u_2^{\sigma +1} \right)
			+ t_2^{\sigma+1} \left( v_1 + v_2^{\sigma+1} \right) 
			+ t_1^{\sigma} t_2 u_2^{\sigma} v_2, t_1u_2 + t_2v_2 )
			
			\cdot x_3( t_1^{\sigma+1} \left( u_2^{\sigma +2} + u_1 u_2 + u_1^{\sigma} \right) 
			+ t_2^{\sigma+1} \left( v_2^{\sigma+2} + v_1 v_2 + v_1^{\sigma} \right)
			+ t_1^{\sigma} t_2 u_1 v_2 )
			
			\cdot x_4( t_1^{\sigma+2} \left( u_1^{\sigma} u_2^{\sigma+1} + u_1^2 u_2 + u_2^{2\sigma +3} \right) 
			+ t_1 \left( u_1^{\sigma} + u_1u_2 + u_2^{\sigma+2} \right) \left( v_1 + v_2^{\sigma+1} \right) t_2^{\sigma+1}
			+ t_1^{\sigma+1} \left( u_1^{\sigma} u_2^{\sigma} + u_2^{2\sigma +2} + u_1^2 \right) t_2v_2
			+ t_1 \left( u_1 + u_2^{\sigma+1} \right) \left( v_2^{\sigma+2} + v_1 v_2 + v_1^{\sigma} \right) t_2^{\sigma+1}
			+ t_2^{\sigma+2} \left( v_1^{\sigma} v_2^{\sigma+1} + v_1^2 v_2 + v_2^{2\sigma +3} \right) 
			+ t_1^2 t_2^{\sigma} u_1^{\sigma} v_1, t_1^{\sigma}u_1 + t_2^{\sigma} v_1 )
			
			\cdot x_5( t_1^{\sigma+1} \left( u_1^{\sigma} u_2^{\sigma} + u_2^{2\sigma+2} + u_1^2 \right) 
			+ t_1 \left( u_1^{\sigma} + u_1 u_2 + u_2^{\sigma+2} \right)t_2^{\sigma} v_2^{\sigma}
			+ t_1 \left( u_1 + u_2^{\sigma+1} \right) t_2^{\sigma}v_1 
			+ t_2^{\sigma+1} \left( v_1^{\sigma} v_2^{\sigma} + v_2^{2\sigma +2} + v_1^2 \right) )
			
			\cdot x_6( t_1^{\sigma+1} \left( u_1^2 u_2 + u_1^{\sigma} u_2^{\sigma+1} + u_1^{\sigma+1} \right) 
			+ t_1^{\sigma} \left( u_1^2 + u_1^{\sigma}u_2^{\sigma} + u_2^{2\sigma+2} \right) t_2v_2
			+ t_1 \left( u_1^{\sigma} + u_1u_2 + u_2^{\sigma+2} \right) t_2^{\sigma} v_1 
			+ t_2^{\sigma+1} \left( v_1^2 v_2 + v_1^{\sigma} v_2^{\sigma+1} + v_1^{\sigma+1} \right)
			+ t_1^{\sigma} t_2 \left( u_1^{\sigma} + u_2^{\sigma +2} \right) \left( v_1 + v_2^{\sigma +1} \right), t_1 \left( u_1 + u_2^{\sigma+1} \right) + t_2 \left( v_1 + v_2^{\sigma+1} \right) )
			
			\cdot x_7( t_1 \left( u_1^{\sigma} + u_1u_2 + u_2^{\sigma+2} \right) 
			+ t_2 \left( v_1^{\sigma} + v_1v_2 + v_2^{\sigma+2} \right) )
			\end{autobreak}
		\end{align*}
	\end{enumerate}
\end{lemma}
\begin{proof}
	At first we state all non-trivial commutator relations we need which are not stated in \cite{TW02}:
	\begin{align*}
	[x_1(t), x_8(u)] &= x_2(t^{\sigma +1}u) x_3(t^{\sigma +1} u^{\sigma}) y_4(t^{\sigma}u) x_5( t^{\sigma +1}u^2 ) y_6(tu)^{-1} x_7( tu^{\sigma} ) \\
	&= x_2(t^{\sigma +1}u) x_3(t^{\sigma +1} u^{\sigma}) y_4(t^{\sigma}u) x_5( t^{\sigma +1}u^2 ) x_6(t^{\sigma +1} u^{\sigma +1})y_6(tu) x_7( tu^{\sigma} ) \\
	[x_7(t), x_2(u)] &= x_4(tu)	\\
	[x_5(t), y_2(u)] &= x_4(tu) \\
	[x_1(t), y_6(u)] &= [x_1(t), x_6(u^{\sigma +1})y_6(u)^{-1}] = x_2(t^{\sigma}u) x_3(tu^{\sigma})x_4(tu^{\sigma +1}) [x_1(t), x_6(u^{\sigma +1})]^{y_6(u)^{-1}} \\
	&= x_2(t^{\sigma}u) x_3(tu^{\sigma})x_4(tu^{\sigma +1}) x_4(tu^{\sigma +1}) = x_2(t^{\sigma}u) x_3(tu^{\sigma})\\	
	[x_7(t), y_2(u)] &= x_6(t^{\sigma}u) x_5(tu^{\sigma}) = x_5(tu^{\sigma}) x_6(t^{\sigma}u) \\
	[y_6(u), x_3(t)] &= x_4(tu)^{-1} = x_4(tu) \\
	[y_6(u), y_4(t)] &= x_5(tu)^{-1} = x_5(tu) \\
	[x_7(t), y_4(u)] &= x_6(tu) \\
	[y_6(t)^{-1}, y_8(u)^{-1}] &= [y_6(t), y_8(u)] = x_7(tu) \\
	[x_5(t), y_8(u)^{-1}] &= [x_5(t), y_8(u)] = x_6(tu) \\
	[x_3(t), y_8(u)^{-1}] &= x_4(t^{\sigma}u) x_5(tu^{\sigma}) x_6(tu^{\sigma+1})
	\end{align*}
	We put $g := x_2(a_2) y_2(b_2) x_3(a_3) x_4(a_4) y_4(b_4) x_5(a_5) x_6(a_6) y_6(b_6) x_7(a_7)$. Using the commutator relations in \cite{TW02} we obtain the following:
	\begin{align*}
	g x_2(z) &= x_2(a_2 + z) y_2(b_2) x_3(a_3) x_4(a_4) y_4(b_4) x_5(a_5) x_6(a_6) y_6(b_6) x_7(a_7) [x_7(a_7), x_2(z)] \\
	&= x_2(a_2 + z) y_2(b_2) x_3(a_3) x_4(a_4) y_4(b_4) x_5(a_5) x_6(a_6) y_6(b_6) x_7(a_7) x_4(a_7 z) \\
	&= x_2(a_2 + z) y_2(b_2) x_3(a_3) x_4(a_4 + a_7 z) y_4(b_4) x_5(a_5) x_6(a_6) y_6(b_6) x_7(a_7) \\
	g y_2(z) &= x_2(a_2) y_2(b_2)y_2(z) x_3(a_3)x_4(a_4) y_4(b_4) [y_4(b_4), y_2(z)] x_5(a_5) [x_5(a_5), y_2(z)] x_6(a_6) y_6(b_6) \\
	&\hspace{0.5cm} \cdot x_7(a_7) [x_7(a_7), y_2(z)] \\
	&= x_2(a_2 + b_2^{\sigma}z) y_2(b_2 + z) x_3(a_3)x_4(a_4) y_4(b_4) x_3(b_4 z)^{-1} x_5(a_5) x_4(a_5 z) x_6(a_6) y_6(b_6) \\
	&\hspace{0.5cm} \cdot x_7(a_7) x_5(a_7 z^{\sigma}) x_6(a_7^{\sigma} z) \\
	&= x_2(a_2 + b_2^{\sigma}z) y_2(b_2 + z) x_3(a_3 + b_4 z) x_4(a_4 + a_5 z) y_4(b_4) x_5(a_5 + a_7 z^{\sigma}) x_6(a_6 + a_7^{\sigma} z) y_6(b_6) x_7(a_7) \\
	g x_3(z) &= x_2(a_2) y_2(b_2) x_3(a_3 + z) x_4(a_4) y_4(b_4) x_5(a_5) x_6(a_6) y_6(b_6) [y_6(b_6), x_3(z)] x_7(a_7) \\
	&= x_2(a_2) y_2(b_2) x_3(a_3 + z) x_4(a_4) y_4(b_4) x_5(a_5) x_6(a_6) y_6(b_6) x_4(b_6 z) x_7(a_7) \\
	&= x_2(a_2) y_2(b_2) x_3(a_3 + z) x_4(a_4 + b_6 z) y_4(b_4) x_5(a_5) x_6(a_6) y_6(b_6) x_7(a_7) \\
	g x_4(z) &= x_2(a_2) y_2(b_2) x_3(a_3) x_4(a_4 + z) y_4(b_4) x_5(a_5) x_6(a_6) y_6(b_6) x_7(a_7) \\
	g y_4(z) &= x_2(a_2) y_2(b_2) x_3(a_3) x_4(a_4) y_4(b_4) y_4(z) x_5(a_5) x_6(a_6) y_6(b_6) [y_6(b_6), y_4(z)] x_7(a_7) [x_7(a_7), y_4(z)] \\
	&= x_2(a_2) y_2(b_2) x_3(a_3) x_4(a_4 + b_4^{\sigma}z) y_4(b_4+z) x_5(a_5) x_6(a_6) y_6(b_6) x_5(b_6 z) x_7(a_7) x_6(a_7 z) \\
	&=  x_2(a_2) y_2(b_2) x_3(a_3) x_4(a_4 + b_4^{\sigma}z) y_4(b_4+z) x_5(a_5 + b_6 z) x_6(a_6 + a_7 z) y_6(b_6) x_7(a_7) \\
	g x_5(z) &= x_2(a_2) y_2(b_2) x_3(a_3) x_4(a_4) y_4(b_4) x_5(a_5 + z) x_6(a_6) y_6(b_6) x_7(a_7) \\
	g x_6(z) &= x_2(a_2) y_2(b_2) x_3(a_3) x_4(a_4) y_4(b_4) x_5(a_5) x_6(a_6 + z) y_6(b_6) x_7(a_7) \\
	gy_6(z) &= x_2(a_2) y_2(b_2) x_3(a_3) x_4(a_4) y_4(b_4) x_5(a_5) x_6(a_6) y_6(b_6)y_6(z) x_7(a_7) \\
	&= x_2(a_2) y_2(b_2) x_3(a_3) x_4(a_4) y_4(b_4) x_5(a_5) x_6(a_6 + b_6^{\sigma}z) y_6(b_6 + z) x_7(a_7) \\
	g x_7(z) &= x_2(a_2) y_2(b_2) x_3(a_3) x_4(a_4) y_4(b_4) x_5(a_5) x_6(a_6) y_6(b_6) x_7(a_7 + z)
	\end{align*}
	
	The rest of the proof is applying the previous relations. We let $u := u_1 + u_2^{\sigma +1}$.
	\begin{align*}
	\begin{autobreak}
		[x_1(t), x_8(u)]^{y_8(u_2)^{-1}} 
	= y_8(u_2) x_2(t^{\sigma +1} u) x_3(t^{\sigma +1} u^{\sigma}) y_4(t^{\sigma} u) x_5(t^{\sigma +1} u^2) y_6(tu)^{-1} x_7(tu^{\sigma}) y_8(u_2)^{-1} 
	= y_8(u_2) x_2(t^{\sigma +1} u) x_3(t^{\sigma +1} u^{\sigma}) y_4(t^{\sigma} u)
		y_8(u_2)^{-1} x_5(t^{\sigma +1} u^2) [x_5(t^{\sigma +1} u^2), y_8(u_2)^{-1}] y_6(tu)^{-1} [y_6(tu)^{-1}, y_8(u_2)^{-1}] x_7(tu^{\sigma}) 
	= x_2(t^{\sigma +1} u) [x_2(t^{\sigma +1}u), y_8(u_2)^{-1}] x_3(t^{\sigma +1} u^{\sigma}) [x_3(t^{\sigma +1} u^{\sigma}), y_8(u_2)^{-1}] y_4(t^{\sigma} u)
		x_5(t^{\sigma +1} u^2) x_6(t^{\sigma +1}u^2 u_2) y_6(tu)^{-1} x_7(tuu_2) x_7(tu^{\sigma}) 
	= x_2(t^{\sigma +1} u) x_3(t^{\sigma+1}uu_2) x_4((t^{\sigma+1}u)^{\sigma} u_2^{\sigma+1}) x_6(t^{\sigma+1}u u_2^{\sigma +2}) 
		x_3(t^{\sigma +1} u^{\sigma}) x_4((t^{\sigma+1}u^{\sigma})^{\sigma} u_2) x_5(t^{\sigma+1}u^{\sigma} u_2^{\sigma}) x_6(t^{\sigma+1}u^{\sigma} u_2^{\sigma+1}) 
		y_4(t^{\sigma} u) x_5(t^{\sigma +1} u^2) x_6(t^{\sigma +1}u^2 u_2) y_6(tu)^{-1} x_7(tuu_2 + tu^{\sigma}) 
	= x_2(t^{\sigma +1} u) x_3(t^{\sigma+1}uu_2 + t^{\sigma +1} u^{\sigma}) x_4((t^{\sigma+1}u)^{\sigma} u_2^{\sigma+1} + (t^{\sigma+1}u^{\sigma})^{\sigma} u_2) y_4(t^{\sigma} u) 
		x_5(t^{\sigma+1}u^{\sigma} u_2^{\sigma} + t^{\sigma +1} u^2) x_6(t^{\sigma+1}u u_2^{\sigma +2} + t^{\sigma+1}u^{\sigma} u_2^{\sigma+1} + t^{\sigma +1}u^2 u_2 + t^{\sigma +1} u^{\sigma +1}) 
		y_6(tu) x_7(tuu_2 + tu^{\sigma}) 
	\end{autobreak}
	\end{align*}
	\begin{align*}
	\begin{autobreak}
		[x_1(t), x_8(u_1)y_8(u_2)] 
	= [x_1(t), x_8(u_1 + u_2^{\sigma +1}) y_8(u_2)^{-1}] 
	= [x_1(t), y_8(u_2)^{-1}] [x_1(t), x_8(u)]^{y_8(u_2)^{-1}}		
	= y_2(tu_2) x_3(t^{\sigma+1} u_2^{\sigma+2}) x_4(t^{\sigma +2} u_2^{2\sigma +3}) y_4(t^{\sigma} u_2^{\sigma +1}) 
		x_5(t^{\sigma +1} u_2^{2\sigma +2}) x_6(t^{\sigma +1} u_2^{2\sigma +3}) x_7(tu_2^{\sigma +2}) x_2(t^{\sigma+1} u) \cdots 
	= x_2(t^{\sigma+1} u) y_2(tu_2) x_3(t^{\sigma+1} u_2^{\sigma+2}) x_4(t^{\sigma +2} u_2^{2\sigma +3} + tu_2^{\sigma +2} t^{\sigma+1} u) y_4(t^{\sigma} u_2^{\sigma +1}) 
		x_5(t^{\sigma +1} u_2^{2\sigma +2}) x_6(t^{\sigma +1} u_2^{2\sigma +3}) x_7(tu_2^{\sigma +2}) x_3(t^{\sigma+1}uu_2 + t^{\sigma+1}u^{\sigma}) \cdots 
	= \cdots x_3(t^{\sigma+1} u_2^{\sigma+2} + t^{\sigma+1}uu_2 + t^{\sigma+1}u^{\sigma}) x_4(t^{\sigma +2} u_2^{2\sigma +3} + tu_2^{\sigma +2} t^{\sigma+1} u) y_4(t^{\sigma} u_2^{\sigma +1}) 
		x_5(t^{\sigma +1} u_2^{2\sigma +2}) x_6(t^{\sigma +1} u_2^{2\sigma +3}) x_7(tu_2^{\sigma +2}) x_4((t^{\sigma+1}u)^{\sigma} u_2^{\sigma+1} + (t^{\sigma+1}u^{\sigma})^{\sigma} u_2) \cdots 
	= \cdots x_4(t^{\sigma +2} u_2^{2\sigma +3} + tu_2^{\sigma +2} t^{\sigma+1} u + (t^{\sigma+1}u)^{\sigma} u_2^{\sigma+1} + (t^{\sigma+1}u^{\sigma})^{\sigma} u_2) y_4(t^{\sigma} u_2^{\sigma +1}) 
		x_5(t^{\sigma +1} u_2^{2\sigma +2}) x_6(t^{\sigma +1} u_2^{2\sigma +3}) x_7(tu_2^{\sigma +2}) y_4(t^{\sigma}u) \cdots 
	= \cdots x_4(t^{\sigma +2} u_2^{2\sigma +3} + tu_2^{\sigma +2} t^{\sigma+1} u + (t^{\sigma+1}u)^{\sigma} u_2^{\sigma+1} + (t^{\sigma+1}u^{\sigma})^{\sigma} u_2 + (t^{\sigma} u_2^{\sigma +1})^{\sigma} t^{\sigma}u)
		y_4(t^{\sigma} u_2^{\sigma +1} + t^{\sigma}u) x_5(t^{\sigma +1} u_2^{2\sigma +2}) x_6(t^{\sigma +1} u_2^{2\sigma +3} + tu_2^{\sigma+2}t^{\sigma}u) x_7(tu_2^{\sigma +2}) 
		x_5(t^{\sigma+1}u^{\sigma} u_2^{\sigma} + t^{\sigma +1} u^2) \cdots 
	= \cdots x_5(t^{\sigma +1} u_2^{2\sigma +2} + t^{\sigma+1}u^{\sigma} u_2^{\sigma} + t^{\sigma +1} u^2) x_6(t^{\sigma +1} u_2^{2\sigma +3} + tu_2^{\sigma+2}t^{\sigma}u) x_7(tu_2^{\sigma +2})
		 x_6(t^{\sigma+1}u u_2^{\sigma +2} + t^{\sigma+1}u^{\sigma} u_2^{\sigma+1} + t^{\sigma +1}u^2 u_2 + t^{\sigma +1} u^{\sigma +1}) \cdots 
	= \cdots x_6(t^{\sigma +1} u_2^{2\sigma +3} + tu_2^{\sigma+2} t^{\sigma}u + t^{\sigma+1}u u_2^{\sigma +2} + t^{\sigma+1}u^{\sigma} u_2^{\sigma+1} + t^{\sigma +1}u^2 u_2 + t^{\sigma +1} u^{\sigma +1})
		x_7(tu_2^{\sigma +2}) y_6(tu) \cdots
	= \cdots y_6(tu) x_7(tu_2^{\sigma +2}) x_7(tuu_2 + tu^{\sigma}) 
	= \cdots x_7( tu_2^{\sigma +2} + tuu_2 + tu^{\sigma} )
	\end{autobreak}	
	\end{align*}
	These relations yields us:
	\begin{align*}
	\begin{autobreak}
		[x_1(t), x_8(u_1, u_2)] 
		= [x_1(t), x_8(u_1)y_8(u_2)] 
		= x_2( t^{\sigma+1} u_1 + t^{\sigma+1} u_2^{\sigma+1} ) y_2(tu_2)		
		\qquad \cdot x_3( t^{\sigma+1} u_2^{\sigma +2} + t^{\sigma+1} u_1 u_2 + t^{\sigma+1} u_2^{\sigma+2} + t^{\sigma+1}u_1^{\sigma} + t^{\sigma+1} u_2^{\sigma+2} )		 
		\qquad \cdot x_4( t^{\sigma+2} u_2^{2\sigma+3} + t^{\sigma +2}u_1 u_2^{\sigma+2} + t^{\sigma+2} u_2^{2\sigma +3} + t^{\sigma+2}u_1^{\sigma}u_2^{\sigma+1} + t^{\sigma+2} u_2^{2\sigma +3}		 
		\qquad \qquad + t^{\sigma +2}u_1^2 u_2 + t^{\sigma+2} u_2^{2\sigma +3} + t^{\sigma+2}u_1 u_2^{\sigma +2} + t^{\sigma +2} u_2^{2\sigma +3} )
		\qquad \cdot y_4(t^{\sigma} u_2^{\sigma+1} + t^{\sigma} u_1 + t^{\sigma}u_2^{\sigma+1})
		\qquad \cdot x_5( t^{\sigma +1} u_2^{2\sigma +2} + t^{\sigma+1}u_1^{\sigma} u_2^{\sigma} + t^{\sigma+1} u_2^{2\sigma +2} + t^{\sigma +1} u_1^2 + t^{\sigma+1} u_2^{2\sigma +2} ) 
		\qquad \cdot x_6( t^{\sigma+1}u_2^{2\sigma +3} +
		t^{\sigma+1}u_1u_2^{\sigma+2} + t^{\sigma+1}u_2^{2\sigma+3}
		+ t^{\sigma+1}u_1u_2^{\sigma+2} + t^{\sigma+1}u_2^{2\sigma+3}
		\qquad \qquad + t^{\sigma+1} u_1^{\sigma} u_2^{\sigma+1} + t^{\sigma+1} u_2^{2\sigma +3} + t^{\sigma+1}u_1^2 u_2 + t^{\sigma+1} u_2^{2\sigma+3}
		\qquad \qquad + t^{\sigma+1}u_1^{\sigma+1} + t^{\sigma+1} u_2^{2\sigma+3} ) y_6(tu_1 + tu_2^{\sigma+1}) 
		\qquad \cdot x_7( tu_2^{\sigma +2} + tu_1u_2 + tu_2^{\sigma +2} + tu_1^{\sigma} + tu_2^{\sigma+2} ) 
		= x_2( t^{\sigma+1} u_1 + t^{\sigma+1} u_2^{\sigma+1}, tu_2 ) 
		x_3( t^{\sigma+1} u_1 u_2 + t^{\sigma+1}u_1^{\sigma} + t^{\sigma+1} u_2^{\sigma +2} )  
		\qquad \cdot x_4( t^{\sigma+2}u_1^{\sigma}u_2^{\sigma+1} + t^{\sigma+2} u_2^{2\sigma +3} + t^{\sigma +2}u_1^2 u_2, t^{\sigma} u_1 ) 
		\qquad \cdot x_5( t^{\sigma+1}u_1^{\sigma} u_2^{\sigma} + t^{\sigma +1} u_1^2 + t^{\sigma+1} u_2^{2\sigma +2} ) 
		\qquad \cdot x_6( t^{\sigma+1} u_1^{\sigma} u_2^{\sigma+1} + t^{\sigma+1} u_1^2 u_2 + t^{\sigma+1}u_1^{\sigma+1}, tu_1 + tu_2^{\sigma+1} ) 
		\qquad \cdot x_7( tu_1u_2 + tu_1^{\sigma} + tu_2^{\sigma +2} )
	\end{autobreak}
	\end{align*}
	This finishes part $(a)$. Now we will prove part $(b)$. We put $g := [ x_1(t_1), x_8(u_1, u_2) ] [x_1(t_2), x_8(v_1, v_2)]$. Then the following hold:
	\begin{align*}
	\begin{autobreak}
		g 
		= x_2( t_1^{\sigma+1} u_1 + t_1^{\sigma +1} u_2^{\sigma+1} + t_2^{\sigma+1} v_1 + t_2^{\sigma +1} v_2^{\sigma+1} + t_1^{\sigma}u_2^{\sigma} t_2 v_2 )
		y_2(t_1 u_2 + t_2v_2)
		x_3( t_1^{\sigma+1} u_1 u_2 + t_1^{\sigma+1}u_1^{\sigma} + t_1^{\sigma+1} u_2^{\sigma +2} + t_1^{\sigma} u_1 t_2v_2 )
		x_4( t_1^{\sigma+2}u_1^{\sigma}u_2^{\sigma+1} 
		+ t_1^{\sigma+2} u_2^{2\sigma +3} + t_1^{\sigma +2}u_1^2 u_2 
		+ \left( t_1 u_1u_2 + t_1 u_1^{\sigma} + t_1 u_2^{\sigma +2} \right)\left( t_2^{\sigma+1} v_1 + t_2^{\sigma +1} v_2^{\sigma+1} \right) 
		+ \left( t_1^{\sigma+1}u_1^{\sigma} u_2^{\sigma} + t_1^{\sigma +1} u_1^2 + t_1^{\sigma+1} u_2^{2\sigma +2} \right) t_2v_2, t_1^{\sigma} u_1 )
		x_5( t_1^{\sigma+1}u_1^{\sigma} u_2^{\sigma} + t_1^{\sigma +1} u_1^2 
		+ t_1^{\sigma+1} u_2^{2\sigma +2} 
		+ \left( t_1 u_1u_2 + t_1 u_1^{\sigma} + t_1 u_2^{\sigma +2} \right) t_2^{\sigma} v_2^{\sigma} )
		x_6( t_1^{\sigma+1} u_1^{\sigma} u_2^{\sigma+1} 
		+ t_1^{\sigma+1} u_1^2 u_2 
		+ t_1^{\sigma+1}u_1^{\sigma+1} 
		+ \left( t_1 u_1u_2 + t_1 u_1^{\sigma} + t_1 u_2^{\sigma +2} \right)^{\sigma} t_2v_2, t_1 u_1 + t_1 u_2^{\sigma+1} )
		x_7( t_1 u_1u_2 
		+ t_1 u_1^{\sigma} 
		+ t_1 u_2^{\sigma +2} )
		x_3( t_2^{\sigma+1} v_1 v_2 + t_2^{\sigma+1}v_1^{\sigma} + t_2^{\sigma+1} v_2^{\sigma +2} ) \cdots 
		= x_2(\ldots) 
		x_3( t_1^{\sigma+1} u_1 u_2 
		+ t_1^{\sigma+1}u_1^{\sigma} 
		+ t_1^{\sigma+1} u_2^{\sigma +2} 
		+ t_1^{\sigma} u_1 t_2v_2 
		+ t_2^{\sigma+1} v_1 v_2 
		+ t_2^{\sigma+1}v_1^{\sigma} 
		+ t_2^{\sigma+1} v_2^{\sigma +2})
		x_4( t_1^{\sigma+2}u_1^{\sigma}u_2^{\sigma+1} 
		+ t_1^{\sigma+2} u_2^{2\sigma +3} 
		+ t_1^{\sigma +2}u_1^2 u_2 
		+ \left( t_1 u_1u_2 + t_1 u_1^{\sigma} + t_1 u_2^{\sigma +2} \right)\left( t_2^{\sigma+1} v_1 + t_2^{\sigma +1} v_2^{\sigma+1} \right)
		+ \left( t_1^{\sigma+1}u_1^{\sigma} u_2^{\sigma} + t_1^{\sigma +1} u_1^2 + t_1^{\sigma+1} u_2^{2\sigma +2} \right) t_2v_2 
		+ \left( t_1u_1 + t_1u_2^{\sigma +1} \right) \left( t_2^{\sigma+1} v_1 v_2 + t_2^{\sigma+1}v_1^{\sigma} + t_2^{\sigma+1} v_2^{\sigma +2} \right), t_1^{\sigma} u_1 )
		x_5( t_1^{\sigma+1}u_1^{\sigma} u_2^{\sigma} 
		+ t_1^{\sigma +1} u_1^2 
		+ t_1^{\sigma+1} u_2^{2\sigma +2} 
		+ \left( t_1 u_1u_2 + t_1 u_1^{\sigma} + t_1 u_2^{\sigma +2} \right) t_2^{\sigma} v_2^{\sigma} )
		x_6( t_1^{\sigma+1} u_1^{\sigma} u_2^{\sigma+1} 
		+ t_1^{\sigma+1} u_1^2 u_2 
		+ t_1^{\sigma+1}u_1^{\sigma+1} 
		+ \left( t_1 u_1u_2 + t_1 u_1^{\sigma} + t_1 u_2^{\sigma +2} \right)^{\sigma} t_2v_2, t_1u_1 + t_1u_2^{\sigma+1} )
		x_7( t_1 u_1u_2 
		+ t_1 u_1^{\sigma} 
		+ t_1 u_2^{\sigma +2} )
		x_4( t_2^{\sigma+2}v_1^{\sigma}v_2^{\sigma+1} 
		+ t_2^{\sigma+2} v_2^{2\sigma +3} 
		+ t_2^{\sigma +2}v_1^2 v_2, t_2^{\sigma} v_1 ) \cdots 
		= x_2(\ldots) x_3(\ldots)
		x_4( t_1^{\sigma+2}u_1^{\sigma}u_2^{\sigma+1} 
		+ t_1^{\sigma+2} u_2^{2\sigma +3} 
		+ t_1^{\sigma +2}u_1^2 u_2 
		+ \left( t_1 u_1u_2 + t_1 u_1^{\sigma} + t_1 u_2^{\sigma +2} \right)\left( t_2^{\sigma+1} v_1 + t_2^{\sigma +1} v_2^{\sigma+1} \right) 
		+ \left( t_1^{\sigma+1}u_1^{\sigma} u_2^{\sigma} + t_1^{\sigma +1} u_1^2 + t_1^{\sigma+1} u_2^{2\sigma +2} \right) t_2v_2
		+ \left( t_1u_1 + t_1u_2^{\sigma +1} \right) \left( t_2^{\sigma+1} v_1 v_2 + t_2^{\sigma+1}v_1^{\sigma} + t_2^{\sigma+1} v_2^{\sigma +2} \right) 
		+ t_2^{\sigma+2}v_1^{\sigma}v_2^{\sigma+1} 
		+ t_2^{\sigma+2} v_2^{2\sigma +3} 
		+ t_2^{\sigma +2}v_1^2 v_2 
		+ \left( t_1^{\sigma} u_1 \right)^{\sigma} t_2^{\sigma} v_1, t_1^{\sigma} u_1 + t_2^{\sigma} v_1 )
		x_5( t_1^{\sigma+1}u_1^{\sigma} u_2^{\sigma} 
		+ t_1^{\sigma +1} u_1^2 
		+ t_1^{\sigma+1} u_2^{2\sigma +2} 
		+ \left( t_1 u_1u_2 + t_1 u_1^{\sigma} + t_1 u_2^{\sigma +2} \right) t_2^{\sigma} v_2^{\sigma} 
		+ \left( t_1 u_1 + t_1 u_2^{\sigma +1} \right) t_2^{\sigma} v_1 )
		x_6( t_1^{\sigma+1} u_1^{\sigma} u_2^{\sigma+1} 
		+ t_1^{\sigma+1} u_1^2 u_2 
		+ t_1^{\sigma+1}u_1^{\sigma+1} 
		+ \left( t_1 u_1u_2 + t_1 u_1^{\sigma} + t_1 u_2^{\sigma +2} \right)^{\sigma} t_2v_2 
		+ \left( t_1 u_1u_2 + t_1 u_1^{\sigma} + t_1 u_2^{\sigma +2} \right) t_2^{\sigma} v_1, t_1u_1 + t_1u_2^{\sigma+1} )
		x_7( t_1 u_1u_2 
		+ t_1 u_1^{\sigma} 
		+ t_1 u_2^{\sigma +2} ) 
		x_5( t_2^{\sigma+1}v_1^{\sigma} v_2^{\sigma} 
		+ t_2^{\sigma +1} v_1^2 
		+ t_2^{\sigma+1} v_2^{2\sigma +2} )
		x_6( t_2^{\sigma+1} v_1^{\sigma} v_2^{\sigma+1} 
		+ t_2^{\sigma+1} v_1^2 v_2 
		+ t_2^{\sigma+1}v_1^{\sigma+1}, t_2v_1 + t_2v_2^{\sigma+1} )
		x_7( t_2v_1v_2 
		+ t_2v_1^{\sigma} 
		+ t_2v_2^{\sigma +2} ) 
		= x_2(\ldots)x_3(\ldots)x_4(\ldots)
		x_5( t_1^{\sigma+1}u_1^{\sigma} u_2^{\sigma} 
		+ t_1^{\sigma +1} u_1^2 
		+ t_1^{\sigma+1} u_2^{2\sigma +2} 
		+ \left( t_1 u_1u_2 + t_1 u_1^{\sigma} + t_1 u_2^{\sigma +2} \right) t_2^{\sigma} v_2^{\sigma} 
		+ \left( t_1 u_1 + t_1 u_2^{\sigma +1} \right) t_2^{\sigma} v_1 
		+ t_2^{\sigma+1}v_1^{\sigma} v_2^{\sigma} 
		+ t_2^{\sigma +1} v_1^2 
		+ t_2^{\sigma+1} v_2^{2\sigma +2} )
		x_6( t_1^{\sigma+1} u_1^{\sigma} u_2^{\sigma+1} 
		+ t_1^{\sigma+1} u_1^2 u_2 
		+ t_1^{\sigma+1}u_1^{\sigma+1} 
		+ \left( t_1 u_1u_2 + t_1 u_1^{\sigma} + t_1 u_2^{\sigma +2} \right)^{\sigma} t_2v_2 
		+ \left( t_1 u_1u_2 + t_1 u_1^{\sigma} + t_1 u_2^{\sigma +2} \right) t_2^{\sigma} v_1 
		+ t_2^{\sigma+1} v_1^{\sigma} v_2^{\sigma+1} 
		+ t_2^{\sigma+1} v_1^2 v_2 
		+ t_2^{\sigma+1}v_1^{\sigma+1} 
		+ \left( t_1u_1 + t_1u_2^{\sigma+1} \right)^{\sigma} \left( t_2v_1 + t_2v_2^{\sigma+1} \right) )
		y_6( t_1u_1 
		+ t_1u_2^{\sigma+1} 
		+ t_2v_1 
		+ t_2v_2^{\sigma+1} )
		x_7( t_1 u_1u_2 
		+ t_1 u_1^{\sigma} 
		+ t_1 u_2^{\sigma +2} 
		+  t_2v_1v_2 
		+ t_2v_1^{\sigma} 
		+ t_2v_2^{\sigma +2} )
	\end{autobreak}
	\qedhere
	\end{align*}
\end{proof}

\bibliography{references}

\begin{thebibliography}{10}

\bibitem{Ab96}
P.~Abramenko.
\newblock {\em Twin buildings and applications to {S}-arithmetic groups},
  volume 1641 of {\em Lecture Notes in Mathematics}.
\newblock Springer-Verlag, Berlin, 1996.

\bibitem{AB08}
P.~Abramenko and K.~S. Brown.
\newblock {\em Buildings}, volume 248 of {\em Graduate Texts in Mathematics}.
\newblock Springer, New York, 2008.
\newblock Theory and applications.

\bibitem{AC16}
D.~Allcock and L.~Carbone.
\newblock Presentation of hyperbolic {K}ac-{M}oody groups over rings.
\newblock {\em J. Algebra}, 445:232--243, 2016.

\bibitem{Bo68}
N.~Bourbaki.
\newblock {\em \'{E}l\'{e}ments de math\'{e}matique. {F}asc. {XXXIV}. {G}roupes
  et alg\`ebres de {L}ie. {C}hapitre {IV}: {G}roupes de {C}oxeter et syst\`emes
  de {T}its. {C}hapitre {V}: {G}roupes engendr\'{e}s par des r\'{e}flexions.
  {C}hapitre {VI}: syst\`emes de racines}.
\newblock Actualit\'{e}s Scientifiques et Industrielles, No. 1337. Hermann,
  Paris, 1968.

\bibitem{CM05}
P.-E. Caprace and B.~M\"{u}hlherr.
\newblock Reflection triangles in {C}oxeter groups and biautomaticity.
\newblock {\em J. Group Theory}, 8(4):467--489, 2005.

\bibitem{CR09}
P.-E. Caprace and B.~R\'{e}my.
\newblock Simplicity and superrigidity of twin building lattices.
\newblock {\em Invent. Math.}, 176(1):169--221, 2009.

\bibitem{Fe98}
A.~A. Felikson.
\newblock Coxeter decompositions of hyperbolic polygons.
\newblock {\em European J. Combin.}, 19(7):801--817, 1998.

\bibitem{He90}
J.-Y. H\'{e}e.
\newblock Construction de groupes tordus en th\'{e}orie de {K}ac-{M}oody.
\newblock {\em C. R. Acad. Sci. Paris S\'{e}r. I Math.}, 310(3):77--80, 1990.

\bibitem{KDis}
D.~Krammer.
\newblock The conjugacy problem for {C}oxeter groups.
\newblock {\em Groups Geom. Dyn.}, 3(1):71--171, 2009.

\bibitem{MR95}
B.~M\"{u}hlherr and M.~Ronan.
\newblock Local to global structure in twin buildings.
\newblock {\em Invent. Math.}, 122(1):71--81, 1995.

\bibitem{MW02}
B.~M\"{u}hlherr and R.~Weidmann.
\newblock Rigidity of skew-angled {C}oxeter groups.
\newblock {\em Adv. Geom.}, 2(4):391--415, 2002.

\bibitem{Ti87}
J.~Tits.
\newblock Uniqueness and presentation of {K}ac-{M}oody groups over fields.
\newblock {\em J. Algebra}, 105(2):542--573, 1987.

\bibitem{Ti92}
J.~Tits.
\newblock Twin buildings and groups of {K}ac-{M}oody type.
\newblock In {\em Groups, combinatorics \& geometry ({D}urham, 1990)}, volume
  165 of {\em London Math. Soc. Lecture Note Ser.}, pages 249--286. Cambridge
  Univ. Press, Cambridge, 1992.

\bibitem{TW02}
J.~Tits and R.~M. Weiss.
\newblock {\em Moufang polygons}.
\newblock Springer Monographs in Mathematics. Springer-Verlag, Berlin, 2002.

\bibitem{Wi20}
T.~Williams.
\newblock Ph{D} thesis (in preparation). {U}niversity of {V}irginia,
  {C}harlottesville.

\end{thebibliography}
\bibliographystyle{abbrv}

\end{document}